\documentclass[preprint]{imsart}
\RequirePackage{amsthm,amsmath,amsfonts,amssymb,dsfont,enumitem}
\RequirePackage[colorlinks,citecolor=blue,urlcolor=blue]{hyperref}
\RequirePackage[OT1]{fontenc}
\RequirePackage[numbers]{natbib}

\startlocaldefs
\numberwithin{equation}{section}

\theoremstyle{plain}
\newtheorem{thm}{Theorem}[section]
\newtheorem{cor}{Corollary}[section]
\newtheorem{lem}{Lemma}[section]
\newtheorem{prop}{Proposition}[section]

\theoremstyle{definition}
\newtheorem{definition}{Definition}[section]
\newtheorem{example}{Example}[section]
\newtheorem{assumption}{Assumption}[section]

\theoremstyle{remark}
\newtheorem{remark}{Remark}[section]

\DeclareMathOperator{\E}{E}
\DeclareMathOperator{\Law}{Law}
\DeclareMathOperator{\Var}{Var}
\DeclareMathOperator{\Cov}{Cov}
\DeclareMathOperator{\Prob}{P}

\DeclareMathOperator{\Unif}{Unif}

\newenvironment{spmatrix}[1]
 {\def\mysubscript{#1}\mathop\bgroup\begin{pmatrix}}
 {\end{pmatrix}\egroup_{\textstyle\text\mysubscript}}

\endlocaldefs

\begin{document}

\begin{frontmatter}

\title{Unifying Sequential Monte Carlo with Resampling Matrices}
\runtitle{Unifying SMC}


\author{\fnms{Robert J.} \snm{Webber}\ead[label=e1]{rw2515@nyu.edu}\thanksref{t1}}
\address{Courant Institute of Mathematical Sciences \\ 
New York University \\ 
251 Mercer St, New York, NY 10012 \\
\phantom{E-mail:\ } \printead{e1}}
\affiliation{New York University}
\runauthor{R. J. Webber}
\thankstext{t1}{Supported by the Advanced Scientific Computing Research Program 
within the DOE Office of Science through award DE-SC0014205
and by NSF RTG award number 1547396}

\begin{abstract}
Sequential Monte Carlo (SMC) is a class of algorithms that approximate high-dimensional expectations of a Markov chain.
SMC algorithms typically include a resampling step.
There are many possible ways to resample,
but the relative advantages of different resampling schemes remain poorly understood.
Here, a theoretical framework for comparing resampling schemes is presented.
The framework uses resampling matrices to provide a simple description
for the SMC resampling step.
The framework identifies the matrix resampling scheme
that gives the lowest possible error.
The framework leads to new asymptotic error formulas
that can be used to compare different 
resampling schemes.
\end{abstract}

\begin{keyword}[class=MSC]
\kwd[Primary ]{65C05}
\kwd[; secondary ]{60J05}
\end{keyword}

\begin{keyword}
\kwd{resampling}
\kwd{Sequential Monte Carlo}
\end{keyword}


\end{frontmatter}


\section{Introduction\label{Introduction}}

Sequential Monte Carlo (SMC) has a history that traces from the 1950's to the present.
The first examples of SMC were simulations of chain polymers in the 1950's \cite{hammersley1954poor,rosenbluth1955monte}.
Starting in the 1960's, SMC was used in the quantum chemistry
community to calculate the ground state energy of the Schr\"{o}dinger equation \cite{kalos1962monte, grimm1971monte}.
SMC became a standard statistical tool in the 1990's, as the
algorithm was applied to problems in Bayesian inference and signal processing \cite{doucetsequential}. 
In recent years, the algorithm continues to fascinate researchers
who are ever developing new variations of SMC algorithms
(e.g., \cite{whiteley2016role, gerber2017negative}).

SMC is a tool for evaluating expectations of the form
\begin{equation*}
\E\left[G_0\left(X_0\right) \prod_{t=0}^{T-1}G_t\left(X_{t-1}, X_t\right)f\left(X_{T-1}, X_T\right)\right]
\end{equation*}
where $\left(X_{t}\right)_{t\geq0}$ is a discrete-time Markov chain
on a sequence of state space $\left(E_t\right)_{t \geq 0}$,
functions $\left(G_{t}\right)_{t\geq0}$ are nonnegative, and $f$ is real-valued.
These expectations are called
Feynman-Kac integrals, and they are notoriously difficult to
evaluate when $T$ is large
\cite{liu2008monte}.
SMC is a sampling algorithm that simulates the dynamics of the Markov chain $\left(X_t\right)_{t \geq 0}$
and provides random
approximations for Feynman-Kac integrals that become increasingly accurate as computational
effort is increased. 

SMC has a wide range of applications from Bayesian statistics to rare event sampling.
In Bayesian contexts, functions $\left(G_{t}\right)_{t\geq0}$ are typically
unnormalized likelihood ratios between prior and posterior distributions.
SMC is used to estimate statistics of the posterior distribution,
and the resulting algorithm is often called the particle filter \cite{doucetsequential}. 
In rare event sampling, on the other hand,
SMC is used to provide estimates of rare event probabilities,
and
functions $\left(G_{t}\right)_{t\geq0}$ bias a process $\left(X_{t}\right)_{t\geq0}$
to explore regions of state space that would rarely be accessed
under typical conditions \cite{hairer2014improved}.




Despite the usefulness of SMC,
practitioners are burdened with the difficult
task of choosing a resampling scheme from the many options.
Past analyses have provided error formulas for a few particular resampling schemes 
(e.g., \cite{del2004feynman, chopin2004central, douc2008limit}).
However, the number of resampling schemes has increased rapidly in recent years \cite{li2015resampling},
and more theoretical analysis is required to rigorously compare schemes.
Error formulas are not available for all common resampling schemes (e.g., stratified resampling),
and there remains no consensus among experts about how best to resample.

One goal of the current paper is to describe the resampling step in a unified way
in order to facilitate analysis.
Thus, Section \ref{introducing} introduces a matrix resampling framework, 
inspired by work of \citet{hu2008basic} and \citet{whiteley2016role}.
Resampling matrices provide a simple description for a great variety of resampling schemes,
and any scheme in the matrix resampling framework is guaranteed to exhibit important convergence behavior.
In particular, Section \ref{introducing} proves unbiasedness, convergence,
and an upper bound on variance for SMC estimates 
made using matrix resampling schemes.

Another goal of the current paper is to present a unified analysis of SMC error.
Section \ref{minimize} explains how error arises within the SMC algorithm
and how error can be reduced by selecting an appropriate resampling scheme.
The scheme that gives the lowest possible resampling error is identified.
To compare the performance of resampling schemes,
Section \ref{minimize} also provides new asymptotic error bounds,
including the first such bounds for stratified resampling and stratified residual resampling.

Technical proofs are presented in an appendix,
following Section \ref{minimize} and the conclusion.



\section{Matrix resampling framework}{\label{introducing}}


The goal of the current section is to provide a matrix resampling framework that ties together diverse SMC resampling schemes.
Section \ref{smcoverview} provides a short overview of SMC.
Section \ref{smcextension} describes the key features of the matrix resampling framework.
Section \ref{simpletheorems} presents convergence theorems
that ensure the validity of SMC estimates.
Section \ref{martingale} presents a martingale argument
to show why SMC estimates are unbiased.


\subsection{Overview of Sequential Monte Carlo}{\label{smcoverview}}



Sequential Monte Carlo begins by sampling initial ``particles'',
and then the algorithm proceeds iteratively through three main steps:
reweighting, resampling, and mutation.
Definition \ref{def:overview} gives an overview of these steps
and the quantities that can be estimated through SMC:

\begin{definition}{\label{def:overview}}
Overview of Sequential Monte Carlo
\begin{enumerate}
\item Initialization: Independently sample
$\xi_{0}^{\left(i\right)}\sim\Law\left(X_{0}\right)$ for $1 \leq i \leq N_0$.

\item
The algorithm proceeds iteratively for $t =0, 1, 2, \ldots$. 

\begin{enumerate}
\item
Reweighting: Assign weights $w_t^{\left(i\right)}$ to each particle $\xi_t^{\left(i\right)}$ with
\begin{equation*}
\begin{cases}
w_t^{\left(i\right)} = G_t\left(\xi_t^{\left(i\right)}\right), & t = 0 \\
w_t^{\left(i\right)} = \hat{w}_{t-1}^{\left(i\right)} 
G_t\left(\hat{\xi}_{t-1}^{\left(i\right)}, \xi_t^{\left(i\right)}\right), & t > 0
\end{cases}
\end{equation*}

\item
Resampling: Replace the ensemble $\left(w_t^{\left(i\right)}, \xi_t^{\left(i\right)}\right)_{1 \leq i \leq N_t}$
with a new ensemble $\left(\hat{w}_t^{\left(j\right)}, \hat{\xi}_t^{\left(j\right)}\right)_{1 \leq j \leq N_{t+1}}$,
where each particle $\hat{\xi}_t^{\left(j\right)}$ is a copy of some particle $\xi_t^{\left(i\right)}$
and weights $\hat{w}_t^{\left(j\right)}$ are defined so that
\begin{equation*}
\frac{1}{N_0} \sum_{i=1}^{N_t} w_t^{\left(i\right)} f\left(\xi_t^{\left(i\right)} \right)
\approx
\frac{1}{N_0} \sum_{j=1}^{N_{t+1}} \hat{w}_t^{\left(j\right)} f\left(\hat{\xi}_t^{\left(j\right)}\right)
\end{equation*}
for all functions $f\colon E_t \rightarrow \mathbb{R}$.

\item Mutation: sample
$\xi_{t+1}^{\left(i\right)}\sim \Law \left(X_{t+1}|X_t=\hat{\xi}_t^{\left(i\right)}\right)$ for $1\leq i \leq N_{t+1}$.
\end{enumerate}

\item
Estimation: To estimate quantities $\E\left[\prod_{s=0}^{t-1} G_s f\right]$, use
\begin{equation*}
\frac{1}{N_0}\sum_{i=1}^{N_t} \hat{w}_{t-1}^{\left(i\right)}
f\left(\hat{\xi}_{t-1}^{\left(i\right)}, \xi_{t}^{\left(i\right)}\right)
\approx \E\left[\prod_{s=0}^{t-1} G_s f\right]
\end{equation*}
\end{enumerate}
\end{definition}


For notational simplicity, in expectations involving the Markov Chain $\left(X_t\right)_{t \geq 0}$,
the arguments of functions will often be omitted.
For example,
$\E\left[\prod_{s=0}^{t-1} G_s f\right]$
denotes $\E\left[G_0\left(X_0\right) \prod_{s=1}^{t-1} G_s\left(X_{s-1},X_s\right) f\left(X_{t-1},X_t\right)\right]$.

%

While the reweighting and mutation steps are straightforward,
there are many different ways to carry out the resampling step.
Outlined below are examples of resampling methods:

\begin{example}[Sequential importance sampling]
The simplest resampling scheme, sequential importance sampling \cite{hammersley1954poor, rosenbluth1955monte},
leaves the ensemble of particles and weights completely unchanged:
\begin{equation*}
\left(\hat{w}_t^{\left(j\right)}, \hat{\xi}_t^{\left(j\right)}\right)_{1 \leq j \leq N_{t+1}}
= \left(w_t^{\left(i\right)}, \xi_t^{\left(i\right)} \right)_{1 \leq i \leq N_t}
\end{equation*}
\end{example}

In sequential importance sampling, weights
$\hat{w}_t^{\left(j\right)}$ 
are multiples of many functions $\left(G_t\right)_{t \geq 0}$:
\begin{equation*}
\hat{w}_t^{\left(j\right)} 
= \hat{w}_{t-1}^{\left(j\right)} G_t\left(\hat{\xi}_{t-1}^{\left(j\right)}, \xi_t^{\left(j\right)}\right) 
= \hat{w}_{t-2}^{\left(j\right)} G_{t-1}\left(\hat{\xi}_{t-2}^{\left(j\right)}, \xi_{t-1}^{\left(j\right)}\right) 
G_t\left(\hat{\xi}_{t-1}^{\left(j\right)}, \xi_t^{\left(j\right)}\right) 
= \cdots
\end{equation*}
Consequently, some weights $\hat{w}_t^{\left(j\right)}$
can be very large, while other weights can be very small.
The imbalance in weights can potentially contribute variance to the estimates
$\frac{1}{N_0}\sum_{i=1}^{N_t} \hat{w}_{t-1}^{\left(i\right)}
f\left(\hat{\xi}_{t-1}^{\left(i\right)}, \xi_{t}^{\left(i\right)}\right)$,
because the single particle with the highest weight can dominate all the others.

Alternatives to sequential importance sampling,
which alleviate the imbalance in weights,
include multinomial resampling and Bernoulli resampling.
\begin{example}[Multinomial resampling]
In multinomial resampling \cite{holland1975adaptation},
updated particles $\left(\hat{\xi}_t^{\left(j\right)}\right)_{1 \leq i \leq N_0}$ are independently sampled with common distribution
\begin{equation*}
\hat{\xi}_t^{\left(j\right)} \sim
\frac{\sum_{i=1}^{N_t} w_t^{\left(i\right)} \delta\left(\xi_t^{\left(i\right)}\right)}
{\sum_{i=1}^{N_t} w_t^{\left(i\right)}}
\end{equation*}
and each updated particle is assigned an updated weight
$\hat{w}_t^{\left(j\right)} = \overline{w}_t = \frac{1}{N_0}\sum_{i=1}^{N_t} w_t^{\left(i\right)}$.
\end{example}
\begin{example}[Bernoulli resampling]
In Bernoulli resampling \cite{kalos1962monte}, 
each of the original particles $\xi_t^{\left(i\right)}$
is replicated $N_t^{\left(i\right)}$ times, where the numbers $N_t^{\left(i\right)}$ are independent random variables with
\begin{equation*}
\begin{cases}
N_t^{\left(i\right)} = \left\lfloor \frac{w_t^{\left(i\right)}}{\overline{w}_t} \right\rfloor + 1,
& \text{with probability} \left\{ \frac{w_t^{\left(i\right)}}{\overline{w}_t} \right\}
\vspace{.1cm} \\
N_t^{\left(i\right)} = \left\lfloor \frac{w_t^{\left(i\right)}}{\overline{w}_t} \right\rfloor,
& \text{otherwise}
\end{cases}
\end{equation*}
Here, the floor function $\left\lfloor \cdot \right\rfloor$ is defined by
$\left\lfloor x\right\rfloor =\max\left\{ z\in\mathbb{Z}:z\leq x\right\}$,
the remainder function $\left\{ \cdot \right\}$ is defined by
$\left\{ x\right\} =x-\left\lfloor x\right\rfloor$,
and $\overline{w}_t= \frac{1}{N_0} \sum_{i=1}^{N_t} w_t^{\left(i\right)}$ is the average of the weights.
After replication, each updated particle $\hat{\xi}_t^{\left(j\right)}$ 
is assigned an updated weight $\hat{w}_t^{\left(j\right)} =\overline{w}_t$.
\end{example}


\subsection{Extending the matrix resampling framework}{\label{smcextension}}

Sequential importance sampling and multinomial resampling are both
\emph{matrix resampling schemes}.
First introduced by \citet{hu2008basic} and \citet{whiteley2016role},
matrix resampling schemes involve a resampling step described by a nonnegative matrix $W_t$
with dimensions $N_t \times N_{t+1}$.
The properties of this matrix are:
\begin{itemize}
\item
The $i$th row sum equals the weight $w_t^{\left(i\right)}$ for $1 \leq i \leq N_t$.
\item
The $j$th column sum equals the updated weight $\hat{w}_t^{\left(j\right)}$ for $1 \leq j \leq N_{t+1}$.
\item
Each updated particle $\hat{\xi}_t^{\left(j\right)}$ is independently drawn from a distribution
determined by the $j$th column of the resampling matrix:
\begin{equation*}
\hat{\xi}_t^{\left(j\right)} \sim \frac{\sum_{i=1}^{N_t} w_t^{\left(i, j\right)} \delta\left(\xi_t^{\left(i\right)}\right)}
{\sum_{i=1}^{N_t} w_t^{\left(i, j\right)}}
\end{equation*}
\end{itemize}

Resampling schemes can be divided into \emph{fixed population} resampling schemes,
where particle numbers $\left(N_t\right)_{t \geq 0}$ are deterministic,
and \emph{random population} resampling schemes, 
where the number of particles 
$\left(N_t\right)_{t \geq 0}$ is random.
While the matrix resampling framework is useful for describing fixed population schemes,
it is necessary to extend the framework 
further in order to describe random population resampling schemes.

This section presents a new extension to the matrix resampling framework 
to random population schemes that satisfy an upper
bound on the maximum possible number of particles $N_t$.
In these schemes, $N_t$ can be bounded by $C_t N_0$ for each $t \geq 0$,
where $\left(C_t\right)_{t \geq 0}$ is a deterministic series of constants.
This assumption is often satisfied for the random population schemes used in practice.
For example, in Bernoulli resampling,
the random numbers $N_t$ satisfy an upper bound $N_t \leq N_0 \left(t + 1\right)$
and cannot grow in an uncontrolled way,
because
\begin{equation*}
N_{t+1} = \sum_{i=1}^{N_t} N_t^{\left(i\right)}
= \sum_{i=1}^{N_t}
\left(\left\lfloor \frac{w_t^{\left(i\right)}}{\overline{w}_t} \right\rfloor 
+ 1\right)
\leq \sum_{i=1}^{N_t} \left(\frac{w_t^{\left(i\right)}}{\overline{w}_t} + 1\right)
= N_0 + N_t
\end{equation*}

The extended matrix resampling framework 
differs from the standard matrix resampling framework by including a ``coffin state'' $c$.
The coffin state is an element of state space
that particles $\hat{\xi}_t^{\left(j\right)}$ can potentially occupy, but
particles in the coffin state do not affect any SMC estimates.
By including a coffin state,
the extended matrix resampling framework
is able to reinterpret many random population schemes
as schemes where the number of particles $\left(N_t\right)_{t \geq 0}$ is deterministic 
but the number of coffin state particles is random.

In the extended matrix resampling framework, 
the Markov chain $X_t$ is allowed to take values in the extended state space $E_t \cup \left\{c\right\}$.
Transitions from the coffin state are described by $\Prob \left\{X_{t + 1} = c \rvert X_t = c \right\} = 1$.
Functions defined on $E_t$ or $E_{t-1} \times E_t$ are extended to take values $f\left(c\right) = 0$ or $f\left(c,c\right) = 0$.
As seen in the definition below, 
the extended matrix resampling framework
includes a row in each resampling matrix $W_t$
governing transitions into the coffin state $c$:

\begin{definition}{\label{def:general}}
Extended matrix resampling framework
\begin{enumerate}

\item Initialization: Independently sample
$\xi_{0}^{\left(i\right)}\sim\Law\left(X_{0}\right)$ for $1 \leq i \leq N_0$.

\item
The algorithm proceeds iteratively for $t =0, 1, 2, \ldots$. 

\begin{enumerate}
\item
Reweighting: Assign weights $w_t^{\left(i\right)}$ to each particle $\xi_t^{\left(i\right)}$ with
\begin{equation*}
\begin{cases}
w_t^{\left(i\right)} = G_t\left(\xi_t^{\left(i\right)}\right), & t = 0 \\
w_t^{\left(i\right)} = \hat{w}_{t-1}^{\left(i\right)} 
G_t\left(\hat{\xi}_{t-1}^{\left(i\right)}, \xi_t^{\left(i\right)}\right), & t > 0
\end{cases}
\end{equation*}

\item Resampling: Select a nonnegative matrix $W_t$ with dimensions $\left(N_t + 1\right) \times N_{t+1}$ 
and row sums 
$\sum_{j=1}^{N_{t+1}} w_t^{\left(i,j\right)} = w_t^{\left(i\right)}$ for $1 \leq i \leq N_t$.
Independently, for $1 \leq j \leq N_{t+1}$, select $\hat{\xi}_t^{\left(j\right)}$ 
from the distribution
\begin{equation*}
\hat{\xi}_t^{\left(j\right)} \sim \frac{\sum_{i=1}^{N_t} w_t^{\left(i, j\right)} \delta\left(\xi_t^{\left(i\right)}\right)
+ w_t^{\left(N_t + 1, j\right)} \delta\left(c\right)}
{\sum_{i=1}^{N_t + 1} w_t^{\left(i, j\right)}}
\end{equation*}
Define the $\hat{w}_t^{\left(j\right)}$ by the column sum
$\hat{w}_t^{\left(j\right)} = \sum_{i=1}^{N_t + 1} w_t^{\left(i, j\right)}$.
\item Mutation: sample
$\xi_{t+1}^{\left(i\right)}\sim \Law \left(X_{t+1}|X_t=\hat{\xi}_t^{\left(i\right)}\right)$ for $1\leq i \leq N_{t+1}$.
\end{enumerate}

\item
Estimation: To estimate quantities $\E\left[\prod_{s=0}^{t-1} G_s f\right]$, use
\begin{equation*}
\frac{1}{N_0}\sum_{i=1}^{N_t} \hat{w}_{t-1}^{\left(i\right)}
f\left(\hat{\xi}_{t-1}^{\left(i\right)}, \xi_{t}^{\left(i\right)}\right)
\approx \E\left[\prod_{s=0}^{t-1} G_s f\right]
\end{equation*}
\end{enumerate}
\end{definition}

The extended matrix resampling framework
encompasses a variety of resampling schemes.
For example, Figure \ref{figure1} presents resampling matrices $W_t$
that correspond to sequential importance sampling, multinomial resampling, and Bernoulli resampling.
In the extended matrix resampling framework,
the choice of which matrix $W_t$ to use can be made adaptively,
incorporating any information, such as the values of particles $\left(\xi_t^{\left(i\right)}\right)_{1 \leq i \leq N_t}$
and their weights $\left(w_t^{\left(i\right)}\right)_{1 \leq i \leq N_t}$.
Only the numbers $\left(N_t\right)_{t \geq 0}$ must be fixed in advance of running the SMC algorithm.

\begingroup\abovedisplayskip=0pt\belowdisplayskip=0pt
\begin{figure}[!htbp]
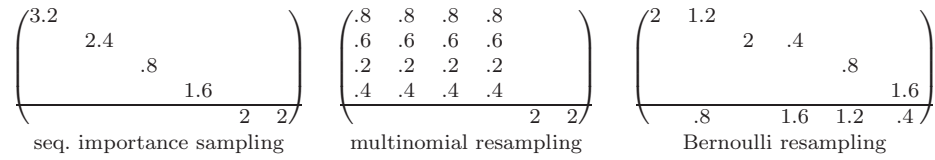

\caption{Examples of resampling matrices $W_0$ when $N_0 = 4$, $N_1 = 6$, and
particles have weights $w_0^{\left(1\right)}=3.2$, $w_0^{\left(2\right)}=2.4$,
$w_0^{\left(3\right)}=.8$, and $w_0^{\left(4\right)}=1.6$.
A horizontal line separates the coffin state $c$.}
\label{figure1}
\begin{minipage}{.32 \textwidth}
\begin{equation*}
\begin{spmatrix}{seq. importance sampling}
3.2 & ~ & ~ & ~ & ~ & ~ \\
~ & 2.4 & ~ & ~ & ~ & ~ \\
~ & ~ & .8 & ~ & ~ & ~ \\
~ & ~ & ~ & 1.6 & ~ & ~\\
\hline
~ & ~ & ~ & ~ & 2 & 2 \\
\end{spmatrix}
\end{equation*}
\end{minipage}
\hfill{}
\begin{minipage}{.32 \textwidth}
\begin{equation*}
\begin{spmatrix}{multinomial resampling}
.8 & .8 & .8 & .8 & ~ & ~ \\
.6 & .6 & .6 & .6 & ~ & ~ \\
.2 & .2 & .2 & .2 & ~ & ~ \\
.4 & .4 & .4 & .4 & ~ & ~ \\
\hline
~ & ~ & ~ & ~ & 2 & 2 \\
\end{spmatrix}
\end{equation*}
\end{minipage}
\hfill{}
\begin{minipage}{.32 \textwidth}
\begin{equation*}
\begin{spmatrix}{Bernoulli resampling}
2 & 1.2 & ~ & ~ & ~ & ~ \\
~ & ~ & 2 & .4 & ~ & ~ \\
~ & ~ & ~ & ~ & .8 & ~\\
~ & ~ & ~ & ~ & ~ & 1.6 \\
\hline
~ & .8 & ~ & 1.6 & 1.2 & .4 \\
\end{spmatrix}
\end{equation*}
\end{minipage}
\end{figure}
\endgroup


\subsection{Unbiasedness, convergence, and variance}{\label{simpletheorems}}

The matrix resampling framework leads to a series of 
powerful results on the unbiasedness, convergence, and variance of SMC estimates.
While versions of these theorems were proved previously
\cite{del2004feynman, douc2008limit, whiteley2016role},
this section presents results that hold more broadly
and include all schemes in the matrix resampling framework.

The first of the key theorems that govern the validity of SMC estimates
ensures that estimates
are unbiased:

\begin{thm}{\label{earlyunbiased}}
If $\E\left|\prod_{t=0}^{T-1} G_t f \right| < \infty$, then SMC estimates are unbiased:
\begin{equation*}
\E\left[\frac{1}{N_{0}}\sum_{i=1}^{N_{T}}\hat{w}_{T-1}^{\left(j\right)}
f\left(\hat{\xi}_{T-1}^{\left(i\right)},\xi_{T}^{\left(i\right)}\right)\right]
=\E\left[\prod_{t=0}^{T-1} G_t f \right]
\end{equation*}
\end{thm}

Theorem \ref{earlyunbiased} is quite general and holds without any additional assumptions.
In contrast, Theorems \ref{weak} and \ref{bounded} will require a mild assumption on the
numbers $\left(N_t\right)_{t \geq 0}$ and on the resampling matrices 
$\left(W_t\right)_{t \geq 0}$:
\begin{assumption}{\label{assumption1}}
There exist absolute constants $\left(C_t\right)_{t \geq 0}$
such that 
$\frac{N_t}{N_0} \leq C_t$ and 
$\max_{1\leq i\leq N_{t+1}}\hat{w}_{t}^{\left(i\right)} \leq C_t \max_{1\leq j\leq N_t}w_{t}^{\left(j\right)}$.
\end{assumption}
Assumption \ref{assumption1} guarantees
that the number of particles does not grow too high
and also that the maximum weight does not grow too high during resampling.
This assumption is satisfied for all the schemes presented in the current paper,
taking $C_t = 1$ for fixed population schemes
and $C_t = t + 1$ for random population schemes.

The next result is a widely useful convergence theorem for SMC estimates:

\begin{thm}{\label{weak}}
If $\E\left[\prod_{s=0}^{t} G_s \right]<\infty$ for $0 \leq t \leq T-1$
and $\E\left|\prod_{t=0}^{T-1} G_t f \right|<\infty$, then
\begin{align*}
\label{weakresult}
& \frac{1}{N_{0}}\sum_{i=1}^{N_{T}}
\hat{w}_{T-1}^{\left(i\right)}f\left(\hat{\xi}_{T-1}^{\left(i\right)},\xi_{T}^{\left(i\right)}\right)
\stackrel{\Prob}{\rightarrow}
\E\left[\prod_{t=0}^{T-1} G_t f \right]
& \text{as } N_0 \rightarrow \infty
\end{align*}
\end{thm}

In Theorem \ref{weak}, it is assumed the SMC algorithm is well-defined on a
probability space $\left(\Omega, \mathcal{F}, \Prob\right)$ for any
number of starting particles $N_0 = 1, 2, \ldots$.
As $N_0 \rightarrow \infty$, Theorem \ref{weak} establishes that SMC 
estimates converge in probabilty to the correct result.

Another key convergence result is a simple upper bound on the variance of SMC estimates.
The upper bound leads to a clear interpretation that SMC estimates
have a $\frac{1}{\sqrt{N_0}}$ error rate
when functions $\left(G_t\right)_{t \geq 0}$ are bounded.

\begin{thm}{\label{bounded}}
If functions $\left(G_t\right)_{0 \leq t \leq T-1}$ are bounded
and $\E\left|\prod_{t=0}^{T-1} G_t f^2 \right|<\infty$, then
\begin{equation*}
\Var\left[\frac{1}{N_{0}}\sum_{i=1}^{N_{T}}
\hat{w}_{T-1}^{\left(i\right)}f\left(\hat{\xi}_{T-1}^{\left(i\right)},\xi_{T}^{\left(i\right)}\right)\right]
\leq \frac{1}{N_0} \E\left[\prod_{t=0}^{T-1} G_t f^2\right] \prod_{t=0}^{T-1} \sup G_t 
\sum_{t=0}^T \prod_{s=0}^{t-1} C_s
\end{equation*}
where $\left(C_t\right)_{t \geq 0}$ 
are the constants appearing in Assumption \ref{assumption1}.
\end{thm}

While antecedents of Theorems \ref{weak} and \ref{bounded} appear in the SMC literature \cite{douc2008limit, whiteley2016role},
the versions presented here are more general with respect to possible resampling schemes
or are more powerful with respect to unbounded functions $f$.
In examples outlined below, 
these theorems determine the convergence behavior
of a diverse set of matrix resampling schemes.
See also Figure \ref{figureexamples},
which provides resampling matrices for the three examples.

\begin{example}[Adaptive resampling and parallel resampling]
Two common variations on the SMC framework are adaptive resampling and parallel resampling.
In adaptive resampling \cite{liu2008monte}, a resampling scheme such as 
multinomial or Bernoulli resampling is triggered 
if the variation in weights exceeds a certain threshold;
otherwise, sequential importance sampling is applied instead.
In parallel resampling \cite{li2015resampling}, 
resampling is applied independently on different processors in order to minimize communication costs.
Theorems \ref{weak} and \ref{bounded} guarantee the convergence of many adaptive and parallel resampling schemes.
In particular, convergence is guaranteed even if
the user decides adaptively which resampling scheme to use at the start of each resampling stage
or if resampling decisions are made in parallel across different machines.
\end{example}

\begin{example}[Pruning and enrichment]
In the pruning and enrichment scheme \cite{grassberger1997pruned},
a lower cutoff $u_t$ and an upper cutoff $U_t$ are selected at the beginning of each resampling step.
If $w_t^{\left(i\right)} > U_t$,
then the particle $\xi_t^{\left(i\right)}$ is split into two replicas $\hat{\xi}_t^{\left(j\right)}$
and $\hat{\xi}_t^{\left(k\right)}$ with reduced weights $\hat{w}_t^{\left(j\right)} = \hat{w}_t^{\left(k\right)} = \frac{1}{2} w_t^{\left(i\right)}$.
If $w_t^{\left(i\right)} < u_t$, then instead
an updated particle $\hat{\xi}^{\left(j\right)}$ is drawn from the distribution 
\begin{equation*}
\hat{\xi}^{\left(j\right)} \sim \frac{1}{2} \delta\left(\xi^{\left(i\right)}\right) + \frac{1}{2} \delta\left(c\right)
\end{equation*}
with weight $\hat{w}^{\left(j\right)} = 2w_t^{\left(i\right)}$.
Lastly, if $u_t \leq w_t^{\left(i\right)} \leq U_t$,
the $i$th particle and weight are left unchanged,
with $\left(\hat{w}_t^{\left(j\right)}, \hat{\xi}_t^{\left(j\right)}\right) 
= \left(w_t^{\left(i\right)}, \hat{\xi}_t^{\left(i\right)}\right)$ for some $1 \leq j \leq N_{t+1}$.
Theorems \ref{weak} and \ref{bounded} guarantee
convergence of the pruning and enrichment scheme even when
cutoff values $u_t$ and $U_t$ are selected adaptively
at the start of each resampling stage.
\end{example}

\begin{example}[Rejection control]
The rejection control scheme \cite{liu1998rejection}
mixes sequential importance sampling and Bernoulli resampling.
In this scheme, first compute the average particle weight $\overline{w}_t = \frac{1}{N_0}\sum_{j=1}^{N_t} w_t^{\left(j\right)}$.
Then, if $w_t^{\left(i\right)} \geq \overline{w}_t$, 
the $i$th particle and weight are left unchanged,
with $\left(\hat{w}_t^{\left(j\right)}, \hat{\xi}_t^{\left(j\right)}\right) 
= \left(w_t^{\left(i\right)}, \hat{\xi}_t^{\left(i\right)}\right)$ for some $1 \leq j \leq N_{t+1}$.
Otherwise, if $w_t^{\left(i\right)} < \overline{w}_t$, 
a particle $\hat{\xi}^{\left(j\right)}$ is drawn from the distribution
\begin{equation*}
\hat{\xi}^{\left(j\right)} \sim \frac{w_t^{\left(i\right)}}{\overline{w}_t} \delta\left(\xi^{\left(i\right)}\right)
+ \left(1 - \frac{w_t^{\left(i\right)}}{\overline{w}_t}\right)  \delta\left(c\right)
\end{equation*}
with weight $\hat{w}^{\left(j\right)}=\overline{w}_t$.
Theorems \ref{weak} and \ref{bounded} are the best known convergence results for the rejection control scheme.
\end{example}

\begingroup\abovedisplayskip=0pt\belowdisplayskip=0pt
\begin{figure}[!htbp]
\caption{Examples of resampling matrices $W_0$ when $N_0 = 4$, $N_1 = 5$, and
particles have weights $w_0^{\left(1\right)}=3.2$, $w_0^{\left(2\right)}=2.4$,
$w_0^{\left(3\right)}=.8$, and $w_0^{\left(4\right)}=1.6$.
In parallel resampling the resampling matrix takes a block diagonal form,
with each block corresponding to a different processor.
In pruning and enrichment, the cutoff values are $c_t = 1$ and $C_t = 3$.}{\label{figureexamples}}
\begin{minipage}{.32 \textwidth}
\begin{equation*}
\begin{spmatrix}{parallel mult. resampling}
1.6 & 1.6 & ~ & ~ & ~ \\
1.2 & 1.2 & ~ & ~ & ~ \\
~ & ~ & .4 & .4 & ~ \\
~ & ~ & .8 & .8 & ~ \\
\hline
~ & ~ & ~ & ~ & 2 \\
\end{spmatrix}
\end{equation*}
\end{minipage}
\hfill{}
\begin{minipage}{.32 \textwidth}
\begin{equation*}
\begin{spmatrix}{pruning and enrichment}
1.6 & 1.6 & ~ & ~ & ~ \\
~ & ~ & 2.4 & ~ & ~ \\
~ & ~ & ~ & .8 & ~ \\
~ & ~ & ~ & ~ & 1.6 \\
\hline
~ & ~ & ~ & .8 & ~ \\
\end{spmatrix}
\end{equation*}
\end{minipage}
\hfill{}
\begin{minipage}{.32 \textwidth}
\begin{equation*}
\begin{spmatrix}{rejection control}
3.2 & ~ & ~ & ~ & ~ \\
~ & 2.4 & ~ & ~ & ~ \\
~ & ~ & .8 & ~ & ~ \\
~ & ~ & ~ & 1.6 & ~ \\
\hline
~ & ~ & 1.2 & .4 & 2 \\
\end{spmatrix}
\end{equation*}
\end{minipage}
\end{figure}
\endgroup

\begin{remark}
Many past analyses of SMC \cite{douc2008limit, chan2013general} 
have focused on SMC estimates of ratios
$\E\left[\prod_{t=0}^{T-1} G_t f \right] \slash \E\left[\prod_{t=0}^{T-1} G_t \right]$.
In the present analysis, the central focus is shifted toward SMC estimates of quantities
$\E\left[\prod_{t=0}^{T-1} G_t f \right]$.
This central focus has three advantages.
First, estimates of $\E\left[\prod_{t=0}^{T-1} G_t f \right]$
are unbiased, making them simpler to analyze than estimates of ratios,
which are typically biased. 
Second, unbiased SMC estimates have not been studied in as much detail
as estimates of ratios have been, despite their central importance in rare event sampling and Bayesian statistics 
\cite{del2004feynman, hairer2014improved}.
Third, convergence properties for estimates of ratios
follow as a corollary of convergence properties for unbiased estimates.
For more details of this relationship, refer to the discussion in the appendix.
\end{remark}


\subsection{Martingale analysis of SMC}{\label{martingale}}

Martingale theory provides an essential tool
for the analysis of SMC
\cite{del2004feynman, douc2008limit, chan2013general}.
In the current section, a martingale is used to show
that SMC estimates are unbiased.
In later sections, the same martingale leads to an 
error decomposition and asymptotic error formulas for SMC estimates.

The first step in a martingale analysis is to define a filtration
and a martingale sequence on that filtration.
Toward this goal, fix functions $\left(G_t\right)_{0 \leq t \leq T-1}$ and $f$
and define $\sigma$-algebras and conditional expectations as follows:

\begin{definition}{\label{sigmadef}}
$\sigma$-algebras and conditional expectations
\begin{enumerate}
\item
Introduce the filtration  $\left(\mathcal{F}_t\right)_{-1 \leq t \leq T}$, where
\begin{equation*}
\begin{cases}
\mathcal{F}_{-1} = \left\{\emptyset, \Omega \right\} \\
\mathcal{F}_0 = \sigma \left( \left(\xi_0^{\left(i\right)}\right)_{1 \leq i \leq N_0}, W_0\right) \\
\mathcal{F}_{t+\frac{1}{2}} = \mathcal{F}_t \vee \sigma \left( \left(\hat{\xi}_t^{\left(j\right)}\right)_{1 \leq j \leq N_{t + 1}} \right), 
& 0 \leq t \leq T-1 \\
\mathcal{F}_{t+1} = \mathcal{F}_{t+\frac{1}{2}} \vee \sigma \left( \left(\xi_{t+1}^{\left(i\right)}\right)_{1 \leq i \leq N_{t+1}}, W_t\right), 
& 0 \leq t \leq T-2 \\
\mathcal{F}_T = \mathcal{F}_{T-\frac{1}{2}} \vee \sigma\left(\left(\xi_T^{\left(i\right)}\right)_{1 \leq i \leq N_T}\right)
\end{cases}
\end{equation*}
Here, $\mathcal{G} \vee \mathcal{H}$ denotes the smallest $\sigma$-algebra
containing $\mathcal{G}$ and $\mathcal{H}$.
\item
Define the conditional expectations
\begin{align*}
& h_t \left(x_t\right) = \E \left[\left. \prod_{s=t+1}^{T-1} G_s f \right\rvert X_t = x_t\right],
& 0 \leq t \leq T-1
\end{align*}
with the convention that $\prod_{\emptyset} G_s = 1$.
\item
To keep the notation simple, 
write $h_T\left(x_T\right) = 1$, $G_T\left(x_{T-1}, x_T\right) = f\left(x_{T-1}, x_T\right)$,
and $w_T^{\left(i\right)} = \hat{w}_{T-1}^{\left(i\right)} 
G_T\left(\hat{\xi}_{T-1}^{\left(i\right)}, \xi_T^{\left(i\right)}\right)$
for $1 \leq i \leq N_T$.
\end{enumerate}
\end{definition}

The next theorem shows that the SMC estimate $\frac{1}{N_0} \sum_{i=1}^{N_T}
\hat{w}_{T-1}^{\left(i\right)} f\left(\hat{\xi}_{T-1}^{\left(i\right)},\xi_T^{\left(i\right)} \right)$
for the quantity $\E\left[\prod_{t=0}^{T-1} G_t f\right]$
can be interpreted as a martingale on the filtration $\mathcal{F}_t$:

\begin{thm}{\label{unbiased}}
If $\E\left|\prod_{t=0}^{T-1} G_t f \right| < \infty$, there exists a martingale $M_t$ on the filtration $\mathcal{F}_t$
that satisfies
\begin{equation*}
\begin{cases}
M_{-1} = \E \left[ \prod_{t=0}^{T-1} G_t f\right] \\
M_t = \frac{1}{N_0} \sum_{i=1}^{N_t} w_t^{\left(i\right)} h_t \left(\xi_t^{\left(i\right)}\right), 
& 0 \leq t \leq T \\
M_{t + \frac{1}{2}} = \frac{1}{N_{0}} \sum_{i=1}^{N_{t+1}} \hat{w}_t^{\left(i\right)}
h_t \left(\hat{\xi}_t^{\left(i\right)}\right), 
& 0 \leq t \leq T-1
\end{cases}
\end{equation*}
\end{thm}
\begin{proof}
For $0 \leq t \leq T-1$,
\begin{align}
\label{startpart1} 
& \quad \E\left[\left.M_{t+1}\right\rvert \mathcal{F}_{t+\frac{1}{2}} \right]
= \E\left[\left.\frac{1}{N_{0}}\sum_{i=1}^{N_{t+1}} 
w_{t+1}^{\left(i\right)} h_{t+1}\left(\xi_{t+1}^{\left(i\right)}\right)
\right\rvert \mathcal{F}_{t+\frac{1}{2}} \right] \\
& = \frac{1}{N_{0}}\sum_{i=1}^{N_{t+1}} \hat{w}_t^{\left(i\right)}
\E\left[\left.
G_{t+1}\left(\hat{\xi}_{t}^{\left(i\right)},\xi_{t+1}^{\left(i\right)}\right)
h_{t+1}\left(\xi_{t}^{\left(i\right)}\right)\right\rvert \mathcal{F}_{t+\frac{1}{2}}\right] \\
\label{endpart1}
& = \frac{1}{N_{0}}\sum_{i=1}^{N_{t+1}}\hat{w}_{t}^{\left(i\right)}
h_{t}\left(\hat{\xi}_{t}^{\left(i\right)}\right)
= M_{t + \frac{1}{2}}
\end{align}
Lines \eqref{startpart1}-\eqref{endpart1} use the fact that $\hat{w}_t^{\left(i\right)}$ is
measurable with respect to $\mathcal{F}_{t+\frac{1}{2}}$, 
as well as the definitions for $w_{t+1}^{\left(i\right)}$, 
$\Law \left(\xi_{t+1}^{\left(i\right)}\rvert \mathcal{F}_{t+\frac{1}{2}}\right)$,
$h_{t+1}$, and $h_t$. 

Next, for $0 \leq t \leq T-1$,
\begin{align}
\label{startpart2} 
& \quad \E\left[\left.M_{t + \frac{1}{2}}\right| \mathcal{F}_t\right] 
= \E\left[\left.\frac{1}{N_{0}}\sum_{j=1}^{N_{t+1}}
\hat{w}_t^{\left(j\right)}h_{t}\left(\hat{\xi}_{t}^{\left(j\right)}\right)
\right\rvert\mathcal{F}_{t}\right] \\
& =\frac{1}{N_{0}}\sum_{j=1}^{N_{t+1}}\hat{w}_t^{\left(j\right)}
\sum_{i=1}^{N_{t}}h_{t}\left(\xi_{t}^{\left(i\right)}\right)
\Prob\left\{\left. \hat{\xi}_{t}^{\left(j\right)} = \xi_t^{\left(i\right)}\right\rvert\mathcal{F}_{t}\right\} \\
\label{endpart2}
& =\frac{1}{N_0} \sum_{j=1}^{N_{t+1}} \hat{w}_t^{\left(j\right)} \sum_{i=1}^{N_t} 
h_{t}\left(\xi_{t}^{\left(j\right)}\right) \frac{w_{t}^{\left(i,j\right)}}{\hat{w}_t^{\left(j\right)}}
=\frac{1}{N_0} \sum_{i=1}^{N_t} w_t^{\left(i\right)}h_{t}\left(\xi_{t}^{\left(i\right)}\right)
= M_t
\end{align}
Lines \eqref{startpart2}-\eqref{endpart2} use the fact that $\hat{w}_t^{\left(j\right)}$
is measurable with respect to $\mathcal{F}_{t}$,
the definition for $\Law\left(\left. \hat{\xi}_{t}^{\left(j\right)}\right\rvert\mathcal{F}_{t}\right)$,
and the fact that $\sum_{j=1}^{N_t} w_t^{\left(i, j\right)} = w_t^{\left(i\right)}$.

Lastly, because $\xi_0^{\left(j\right)} \sim \Law \left(X_0\right)$ for $1 \leq j \leq N_0$,
$\E \left[ \frac{1}{N_0}\sum_{i=1}^{N_0} w_0^{\left(i\right)} h_0\left(\xi_{0}^{\left(i\right)}\right) \right]
= \E \left[G_0 h_0 \right]
= \E \left[ \prod_{t=0}^{T-1} G_t f \right]$.
\end{proof}

Theorem \ref{unbiased} guarantees the unbiasedness of SMC estimates,
confirming Theorem \ref{earlyunbiased}.


\section{Unified analysis of SMC error}{\label{minimize}}

The current section provides a unified analysis of SMC error
which facilitates comparison of different resampling schemes.
Section \ref{efficiency} defines complete resampling schemes,
a subset of matrix resampling schemes
which will be covered in the error analysis.
Section \ref{factors} explains how error arises within the SMC algorithm
and how error can be reduced by selecting an appropriate resampling scheme.
Section \ref{sort} identifies the matrix resampling scheme that gives
the lowest possible error.
Section \ref{asymtheory} presents new asymptotic formulas
that can be used to rigorously compare the error associated
with different resampling schemes.


\subsection{Complete resampling schemes}{\label{efficiency}}

A complete resampling scheme is a matrix resampling scheme with the requirement
that all the updated weights $\hat{w}_t^{\left(j\right)}$ equal the same weight
$\overline{w}_t = \frac{1}{N_0} \sum_{i=1}^{N_t} w_t^{\left(i\right)}$.
Complete resampling schemes, which include Bernoulli resampling and multinomial resampling,
are very prominent in discussions of SMC.
In fact, several previous reviews of resampling methods
focused solely on complete resampling schemes \cite{douc2005comparison, hol2006resampling}.
The error analysis makes the following assumption:
\begin{assumption}{\label{assumption2}}
The resampling scheme is complete; that is, all the updated weights $\hat{w}_t^{\left(j\right)}$ equal the same weight, 
$\overline{w}_t = \frac{1}{N_0} \sum_{i=1}^{N_t} w_t^{\left(i\right)}$.
\end{assumption}

There are two major factors that determine the value of a resampling scheme:
the computational cost of using the scheme and the accuracy of the estimates it provides.
The advantage of analyzing complete resampling schemes is that all complete
resampling schemes share a similar computational cost.
In particular, the computational cost of an SMC algorithm is proportional
to the number of non-coffin particles,
and the next proposition
guarantees that the number of non-coffin particles
is similar for all complete resampling schemes, with a statistical range of $N_0 \pm \sqrt{N_0}$ particles:

\begin{prop}{\label{particlenumber}}
If at least one of the weights $\left(w_t^{\left(i\right)}\right)_{1 \leq i \leq N_0}$
is positive, then the number of non-coffin particles satisfies
\begin{equation*}
\begin{cases}
\E \left[\left. \sum_{j=1}^{N_{t+1}} \mathds{1}\left\{\hat{\xi}_t^{\left(j\right)} \neq c \right\} \right| \mathcal{F}_t\right]  = N_0 \\
\Var \left[\left.\sum_{j=1}^{N_{t+1}} \mathds{1}\left\{\hat{\xi}_t^{\left(j\right)} \neq c \right\} \right| \mathcal{F}_t \right] \leq N_0
\end{cases}
\end{equation*}
\end{prop}
\begin{proof}
Calculate
$\sum_{j=1}^{N_{t+1}} \Prob \left\{\left.\hat{\xi}_t^{\left(j\right)} \neq c \right| \mathcal{F}_t \right\} 
= \sum_{j=1}^{N_{t+1}} \sum_{i=1}^{N_t} \frac{w_t^{\left(i, j\right)}}{\overline{w}_t}
= N_0$
and 
\begin{align*}
& \quad \sum_{j=1}^{N_{t+1}} \Var \left[\left. \mathds{1}\left\{\hat{\xi}_t^{\left(i\right)} \neq c \right\} \right| \mathcal{F}_t \right]
= \sum_{j=1}^{N_{t + 1}} \frac{w_t^{\left(N_t + 1,j\right)}}{\overline{w}_t} \left(1-\frac{w_t^{\left(N_t + 1,j\right)}}{\overline{w}_t}\right) \\
& \leq \sum_{j=1}^{N_{t + 1}}\left(1-\frac{w_t^{\left(N_t + 1, j\right)}}{\overline{w}_t}\right)
= \sum_{j=1}^{N_{t + 1}} \sum_{i = 1}^{N_t} \frac{w_t^{\left(i,j\right)}}{\overline{w}_t} = N_0
\end{align*}
\end{proof}

Since all complete resampling schemes share a similar computational cost,
it is the accuracy of these schemes that should be the determining factor in deciding which scheme to use.
The accuracy of SMC estimates made using various resampling schemes is explored in depth in the subsequent sections.

\subsection{Factors contributing to SMC error}{\label{factors}}

The goal of the current section is to show how each step of the SMC algorithm
contributes error to SMC estimates
and how this error can be reduced by selecting an appropriate resampling scheme.

The starting point for the decomposition of SMC error is the martingale introduced in Theorem \ref{unbiased}.
\begin{equation*}
\begin{cases}
M_{-1} = \E \left[ \prod_{t=0}^{T-1} G_t f\right] \\
M_t = \frac{1}{N_0} \sum_{i=1}^{N_t} w_t^{\left(i\right)} h_t \left(\xi_t^{\left(i\right)}\right), 
& 0 \leq t \leq T \\
M_{t + \frac{1}{2}} = \frac{\overline{w}_t}{N_0} \sum_{i=1}^{N_{t+1}}
h_t \left(\hat{\xi}_t^{\left(i\right)}\right), 
& 0 \leq t \leq T-1
\end{cases}
\end{equation*}
where
$h_t \left(x_t\right) = \E \left[\left. \prod_{s=t+1}^{T-1} G_s f \right\rvert X_t = x_t\right]$.
At time $t = -1$,
the martingale is a perfect estimate $M_{-1} = \E\left[\prod_{t=0}^{T-1} G_t f\right]$.
At time $t = T$,
the martingale has evolved to become an imperfect estimate
$M_T = \frac{\overline{w}_{T-1}}{N_0} \sum_{i=1}^{N_T} f\left(\hat{\xi}_{T-1}^{\left(i\right)}, \xi_T^{\left(i\right)}\right)$.
An additive decomposition of SMC error is
\begin{align*}
\label{thisone}
& \quad \frac{\overline{w}_{T-1}}{N_0} \sum_{i=1}^{N_T} f\left(\hat{\xi}_{T-1}^{\left(i\right)}, \xi_T^{\left(i\right)}\right)
- \E\left[\prod_{t=0}^{T-1} G_t f\right] \\
& = \underbrace{\left(M_0 - M_{-1}\right)}_{\text{initialization error}} + 
\sum_{t=0}^{T-1} \underbrace{\left(M_{t+\frac{1}{2}} - M_t\right)}_{\text{resampling error}}
+ \sum_{t=0}^{T-1} \underbrace{\left(M_{t+1} - M_{t+\frac{1}{2}}\right)}_{\text{mutation error}}
\end{align*}

In this decomposition, SMC error is the sum of three uncorrelated error sources:
initialization error, resampling error and mutation error.
The first error source is initialization error, 
which can be written
\begin{equation*}
M_0 - M_{-1} = \frac{1}{N_0} \sum_{i=1}^{N_0} 
\left\{G_0\left(\xi_0^{\left(i\right)}\right) h_0\left(\xi_0^{\left(i\right)}\right)
- \E\left[G_0 h_0\right]\right\}
\end{equation*}
Intialization error is caused by random sampling of the particles $\left(\xi_0^{\left(i\right)}\right)_{1 \leq i \leq N_0}$
during the initialization step.
The mean squared initialization error can be calculated
\begin{equation*}
\E\left|M_0 - M_{-1}\right|^2 = \Var\left[\frac{1}{N_0}\sum_{i=1}^{N_0} G_0\left(\xi_0^{\left(i\right)}\right) h_0\left(\xi_0^{\left(i\right)}\right)\right] = \frac{1}{N_0}\Var\left[G_0 h_0\right]
\end{equation*}
This error source is the same for all resampling schemes,
with no dependence on the particular resampling scheme that is used.

Similar to initialization error is mutation error.
Mutation error arises from the random sampling of particles $\left(\xi_t^{\left(i\right)}\right)_{1 \leq i \leq N_t}$ during a mutation step.
Mutation error $M_{t+1} - M_{t+\frac{1}{2}}$ can be written
\begin{equation*}
\frac{\overline{w}_t}{N_0}\sum_{i=1}^{N_t}
\left\{G_{t+1}\left(\hat{\xi}_t^{\left(i\right)}, \xi_{t+1}^{\left(i\right)}\right) h_{t+1}\left(\xi_{t+1}^{\left(i\right)}\right)
- \E\left[\left.G_{t+1} h_{t+1} \right|X_t = \hat{\xi}_t^{\left(i\right)}\right]\right\}
\end{equation*}
An asymptotic expansion shows how mutation error approaches a fixed asymptotic limit,
regardless of which resampling scheme is used:
\begin{prop}{\label{decomposition}}
Assume functions $\left(G_t\right)_{0 \leq t \leq T-1}$ are bounded and assume
$\E\left[\prod_{t=0}^{T-1} G_t f^2 \right] < \infty$. 
Then, at each time $0 \leq t \leq T-1$ there exists a constant $C>0$,
independent of resampling scheme, such that
\begin{multline*}
\frac{1}{N_0}\E\left[\prod_{s=0}^t G_s\right] 
\E\left[\prod_{s=0}^t G_s \Var\left[\left.G_{t+1} h_{t+1} \right| X_t\right]\right]
\leq \E\left|M_{t+1} - M_{t + \frac{1}{2}}\right|^2 \\
\leq \frac{1}{N_0} \E\left[\prod_{s=0}^t G_s\right] 
\E\left[\prod_{s=0}^t G_s \Var\left[\left.G_{t+1} h_{t+1} \right| X_t\right]\right] + \frac{C}{N_0^2}
\end{multline*}
\end{prop}
\begin{proof}
Define $\nu\left(x_t\right) = \Var\left[\left.G_{t+1} h_{t+1}\right| X_t = x_t\right]$. By Theorem \ref{bounded},
\begin{align*}
0 & \leq \E \left[ \frac{\overline{w}_t^2}{N_0} \sum_{i=1}^{N_t} \nu\left(\hat{\xi}_t^{\left(i\right)}\right)\right]
- \E \left[\overline{w}_t\right] \E\left[\frac{\hat{w}_t}{N_0} \sum_{i=1}^{N_t} \nu\left(\hat{\xi}_t^{\left(i\right)}\right)\right] \\
& = \Cov\left[\overline{w}_t, \frac{\hat{w}_t}{N_0} \sum_{i=1}^{N_t} \nu\left(\hat{\xi}_t^{\left(i\right)}\right)\right]
\leq \Var\left[\overline{w}_t\right]^{\frac{1}{2}}
\Var \left[\frac{\overline{w}_t}{N_0} \sum_{i=1}^{N_t} \nu\left(\hat{\xi}_t^{\left(i\right)}\right)\right]^{\frac{1}{2}}
\leq \frac{C}{N_0}
\end{align*}
Moreover, Theorem \ref{unbiased} guarantees $\E\left[\overline{w}_t\right] = \E\left[\prod_{s=0}^t G_s\right]$
and 
\begin{equation*}
\E\left[\frac{\overline{w}_t}{N_0} \sum_{i=1}^{N_t} \nu\left(\hat{\xi}_t^{\left(i\right)}\right)\right]
= \E\left[\prod_{s=0}^t G_s \Var\left[\left.G_{t+1} h_{t+1}\right|X_t\right]\right]
\end{equation*}
\end{proof}

In summary, Proposition \ref{decomposition} demonstrates that mutation error,
just like initialization error, does not depend on which particular complete resampling scheme is used.

Having discussed two sources of SMC error -- initialization error and resampling error --
the last error source that remains to be discussed is resampling error.
Resampling error can be written
\begin{equation*}
M_{t+\frac{1}{2}} - M_t
= \frac{\overline{w}_t}{N_0} \sum_{j=1}^{N_{t+1}} h_t\left(\hat{\xi}_t^{\left(j\right)}\right)
- \frac{1}{N_0} \sum_{i=1}^{N_t} h_t\left(\xi_t^{\left(i\right)}\right)
\end{equation*}
Resampling error results from random population changes during the resampling step.
Resampling error exhibits quite different behavior from initialization and mutation error:
the size of this error can vary significantly
depending on which particular resampling scheme is used.

A tool for measuring resampling error \cite{douc2005comparison} is resampling variance
\begin{equation*}
\hat{V}_t^2\left[h_t\right] = \Var\left[\left.\frac{\overline{w}_t}{N_0} \sum_{j=1}^{N_{t+1}} h_t\left(\hat{\xi}_t^{\left(j\right)}\right) 
\right| \mathcal{F}_t\right]
\end{equation*}
Reducing resampling variance is a means toward increasing SMC efficiency.
As illustrated in the next lemma,
resampling variance can be reduced by selecting an appropriate resampling scheme:
\begin{lem}{\label{simplelemma}}
\begin{enumerate}[label = (\alph*)]
\item{\label{representation1}}
Let $h_t \in \mathbb{R}^{N_t + 1}$ denote the vector with $h_t^{\left(i\right)}=h_t\left(\xi^{\left(i\right)}\right)$ for $1 \leq i \leq N_t$ and $h_t^{\left(N_t + 1\right)} = 0$.
Then, resampling variance $\hat{V}_t^2\left[h_t\right]$ can be written as a quadratic function of the resampling matrix $W_t$:
\begin{equation*}
\frac{\overline{w}_t}{N_0^2} \sum_{i=1}^{N_t} w_t^{\left(i\right)}
\left|h_t\left(\xi_t^{\left(i\right)}\right)\right|^2
- \frac{1}{N_0^2} h_t^T W_t W_t^T h_t
\end{equation*}
Consequently, minimizing resampling variance is a concave minimization problem.
\item{\label{representation3}}
Consider a resampling matrix $W_t$ containing a sequence of columns
$c_{j_1}, c_{j_2}, \ldots, c_{j_K}$.
Then, replacing the columns $c_{j_1}, c_{j_2}, \ldots, c_{j_K}$ 
with $K$ identical columns $\frac{1}{K} \sum_{k=1}^K c_{j_k}$
either increases resampling variance or leaves resampling variance unchanged.
\end{enumerate}
\end{lem}
\begin{proof}[Proof of Lemma \ref{simplelemma}]
Resampling variance $\hat{V}^2_t\left[h_t\right]$ can be written as
\begin{align*}
& \quad \Var\left[\left.\frac{\overline{w_t}}{N_0}  \sum_{j=1}^{N_{t+1}}
h_t\left(\hat{\xi}_t^{\left(j\right)}\right) \right| \mathcal{F}_t\right] 
=\frac{\overline{w}_t^2}{N_0^2} \sum_{j=1}^{N_{t+1}} \Var\left[\left.
h_t\left(\hat{\xi}_t^{\left(j\right)}\right) \right| \mathcal{F}_t\right] \\
& = \frac{\overline{w}_t^2}{N_0^2} \sum_{j=1}^{N_{t+1}} 
\E\left[\left.\left|h_t\left(\hat{\xi}_t^{\left(j\right)}\right)\right|^2 \right| \mathcal{F}_t\right]
- \frac{\overline{w}_t^2}{N_0^2} \sum_{j=1}^{N_{t+1}} 
\left|\E\left[\left.h_t\left(\hat{\xi}_t^{\left(j\right)}\right)\right|\mathcal{F}_t\right] \right|^2 \\
& =  \frac{\overline{w}_t}{N_0} \E\left[\left. \frac{\overline{w}_t}{N_0} 
\sum_{j=1}^{N_{t+1}} \left|h_t\left(\hat{\xi}_t^{\left(j\right)}\right)\right|^2 \right|\mathcal{F}_t\right]
- \frac{1}{N_0^2} \sum_{j=1}^{N_{t+1}}
\left|\sum_{i=1}^{N_t} w_t^{\left(i,j\right)} h_t\left(\xi_t^{\left(i\right)}\right) \right|^2 \\
& = \frac{\overline{w}_t}{N_0^2} \sum_{i=1}^{N_t} w_t^{\left(i\right)}
\left|h_t\left(\xi_t^{\left(i\right)}\right)\right|^2
- \frac{1}{N_0^2} h_t^T W_t W_t^T h_t
\end{align*}

Next, let $W_t^{\left(1\right)}, W_t^{\left(2\right)}, \ldots, W_t^{\left(K\right)}$ be the resampling matrices
formed by cyclic permutations of columns $c_{j_1}, c_{j_2}, \ldots, c_{j_K}$.
Then, $h_t^T W_t^{\left(k\right)} {W_t^{\left(k\right)}}^T h_t = h_t^T W_t W_t^T h_t$ for each $1 \leq k \leq K$.
By convexity of $W_t \mapsto h_t^T W_t W_t^T h_t$,
\begin{multline*}
h_t^T \left(\frac{1}{K} \sum_{k=1}^K W_t^{\left(k\right)}\right)
\left(\frac{1}{K} \sum_{k=1}^K W_t^{\left(k\right)}\right)^T h_t \\
\leq \frac{1}{K} \sum_{k=1}^K h_t^T W_t^{\left(k\right)} {W_t^{\left(k\right)}}^T h_t
= h_t^T W_t W_t^T h_t
\end{multline*}
\end{proof}

The second part of Lemma \ref{simplelemma} 
is a useful device for comparing common resampling schemes.
In examples below, the lemma is used to analyze efficiency of
three common resampling schemes:
\emph{stratified}, \emph{multinomial residual},
and \emph{stratified residual} resampling.
See also Figure \ref{figure2}, which provides resampling matrices for these three schemes.

\begin{example}[Stratified resampling]
In stratified resampling \cite{kitagawa1996monte},
sample uniform random variables
$U_t^{\left(j\right)} \sim \Unif \left(\frac{j-1}{N_0},\frac{j}{N_0}\right)$ for $1 \leq j \leq N_0$
and select particles 
$\hat{\xi}_t^{\left(j\right)} = Q_t\left(U_t^{\left(j\right)}\right)$, where
\begin{align*}
& Q_t\left(x\right)=\xi_t^{\left(i\right)},
&
\frac{\sum_{k=1}^{i-1} w_t^{\left(k\right)}}
{\sum_{k=1}^{N_t} w_t^{\left(k\right)}}
\leq x < 
\frac{\sum_{k=1}^i w_t^{\left(k\right)}}
{\sum_{k=1}^{N_t} w_t^{\left(k\right)}}
\end{align*}
It is seen in Figure \ref{figure2} that the resampling matrix for stratified resampling
takes a particular form,
with nonzero matrix entries forming a path rightwards and downwards.
By averaging over all matrix columns, the multinomial resampling matrix is obtained.
Thus, by Lemma \ref{simplelemma}, the resampling variance of stratified resampling
is always as low or lower than that of multinomial resampling.
\end{example}
\begin{example}
In multinomial residual resampling \cite{brindle1980genetic}, first select $\left\lfloor \frac{w_t^{\left(i\right)}}{\overline{w}_t}\right\rfloor$
copies of each particle $\xi_t^{\left(i\right)}$.
Then, select an additional
$R_t = \sum_{i=1}^{N_t} \left\{ \frac{w_t^{\left(i\right)}}{\overline{w}_t}\right\}$ particles 
$\hat{\xi}_t^{\left(j\right)}$
independently from the distribution
\begin{equation*}
\hat{\xi}_t^{\left(j\right)} \sim \frac{1}{R_t} \sum_{i=1}^{N_t}
\left\{ \frac{w_t^{\left(i\right)}}{\overline{w}_t}\right\} \delta\left(\xi_t^{\left(i\right)}\right)
\end{equation*}
It is seen in Figure \ref{figure2} that the resampling matrix for multinomial residual resampling
contains a block of columns with just one nonzero matrix entry per column.
By averaging over all matrix columns, the multinomial resampling matrix is obtained.
Thus, by Lemma \ref{simplelemma}, the resampling variance of stratified resampling
is always as low or lower than that of multinomial resampling.
\end{example}
\begin{example}
Stratified residual resampling \cite{bolic2003new}
combines aspects of stratified resampling and multinomial residual resampling.
First select $\left\lfloor \frac{w_t^{\left(i\right)}}{\overline{w}_t}\right\rfloor$
copies of each particle $\xi_t^{\left(i\right)}$.
Then, for $1 \leq j \leq R_t = \sum_{i=1}^{N_t} \left\{ \frac{w_t^{\left(i\right)}}{\overline{w}_t}\right\}$, 
sample a uniform random variable $U_t^{\left(j\right)} \sim \Unif \left(\frac{j-1}{R_t},\frac{j}{R_t}\right)$ and 
select the particle $\hat{\xi}_t^{\left(j\right)} = Q_t\left(U_t^{\left(j\right)}\right)$,
where
\begin{align*}
& Q_t\left(x\right)=\xi_t^{\left(i\right)},
&
\frac{1}{R_t} \sum_{k=1}^{i-1} \left\{ \frac{w_t^{\left(k\right)}}{\overline{w}_t}\right\}
\leq x < 
\frac{1}{R_t} \sum_{k=1}^i \left\{ \frac{w_t^{\left(k\right)}}{\overline{w}_t}\right\}
\end{align*}
The resampling matrix for stratified residual resampling
contains a block of columns where entries for a path rightwards and downwards.
By averaging over this block of columns, the multinomial residual matrix is obtained.
By Lemma \ref{simplelemma}, the resampling variance of stratified residual resampling
is as low or lower than that of multinomial residual resampling.
\end{example}

\begingroup\abovedisplayskip=0pt\belowdisplayskip=0pt
\begin{figure}[!htbp]
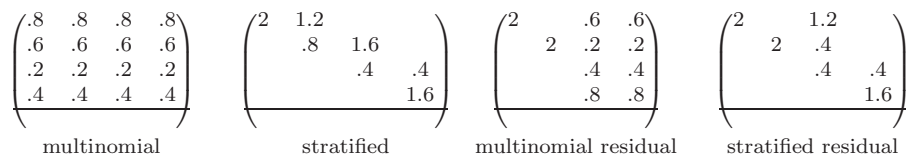

\caption{Examples of resampling matrices $W_0$ when $N_0 = 4$, $N_1 = 4$, and
particles have weights $w_0^{\left(1\right)}=3.2$, $w_0^{\left(2\right)}=2.4$,
$w_0^{\left(3\right)}=.8$, and $w_0^{\left(4\right)}=1.6$.}
\label{figure2}
\begin{minipage}{.22 \textwidth}
\begin{equation*}
\begin{spmatrix}{multinomial}
.8 & .8 & .8 & .8 \\
.6 & .6 & .6 & .6 \\
.2 & .2 & .2 & .2 \\
.4 & .4 & .4 & .4 \\
\hline
~ & ~ & ~ & ~ \\
\end{spmatrix}
\end{equation*}
\end{minipage}
\hfill{}
\begin{minipage}{.22 \textwidth}
\begin{equation*}
\begin{spmatrix}{stratified}
2 & 1.2 & ~ & ~ \\
~ & .8 & 1.6 & ~ \\
~ & ~ & .4 & .4 \\
~ & ~ & ~ & 1.6 \\
\hline
~ & ~ & ~ & ~ \\
\end{spmatrix}
\end{equation*}
\end{minipage}
\hfill{}
\begin{minipage}{.22 \textwidth}
\begin{equation*}
\begin{spmatrix}{multinomial residual}
2 & ~ & .6 & .6 \\
~ & 2 & .2 & .2 \\
~ & ~ & .4 & .4 \\
~ & ~ & .8 & .8 \\
\hline
~ & ~ & ~ & ~ \\
\end{spmatrix}
\end{equation*}
\end{minipage}
\hfill{}
\begin{minipage}{.22 \textwidth}
\begin{equation*}
\begin{spmatrix}{stratified residual}
2 & ~ & 1.2 & ~ \\
~ & 2 & .4 & ~ \\
~ & ~ & .4 & .4 \\
~ & ~ & ~ & 1.6 \\
\hline
~ & ~ & ~ & ~ \\
\end{spmatrix}
\end{equation*}
\end{minipage}
\end{figure}
\endgroup

\begin{remark}
While Proposition \ref{decomposition} requires that functions $\left(G_t\right)_{t \geq 0}$ are bounded,
this assumption can be lifted, at the cost of greater complexity.
Using methods to be presented in Section \ref{asymtheory},
it can be shown that mutation error converges in distribution
\begin{equation*}
\label{asymnormal}
\frac{1}{\sqrt{N_0}}\left(M_{t+1} - M_{t+\frac{1}{2}}\right) \stackrel{\mathcal{D}}{\rightarrow} 
N\left(0, \E\left[\prod_{s=0}^t G_s\right]
\E\left[\prod_{s=0}^t G_s \Var\left[\left.G_{t+1} h_{t+1}\right| X_t\right]\right]\right)
\end{equation*}
whenever the asymptotic variance is finite.
The asymptotic distribution does not depend on which resampling scheme is used.
\end{remark}

\begin{remark}
Similar to the examples above,
\citet{douc2005comparison} compared resampling variance
between different resampling schemes.
But while \cite{douc2005comparison} used explicit resampling variance calculations, 
the resampling matrix framework provides a quicker route to comparing schemes.
In the examples above, it is enough simply to compare columns between resampling matrices
and apply Lemma \ref{simplelemma} to obtain a rigorous error comparison.
\end{remark}



\subsection{Minimizing resampling variance}{\label{sort}}

The goal of SMC is to compute a quantity $\E\left[\prod_{s=t+1}^{T-1} G_s f\right]$ with minimal error.
Sections \ref{efficiency} and \ref{factors} have demonstrated that the error
of an estimate depends critically on the resampling variance.
Thus, it is of foremost concern to find resampling schemes
that minimize resampling variance.



Theorem \ref{optimal} identifies the minimal variance resampling scheme,
a scheme that sorts particles depending on the values
$h_t\left(\xi_t^{\left(i\right)}\right)$:

\begin{thm}{\label{optimal}}
\begin{enumerate}[label = (\alph*)]
\item{\label{complexscheme}}
The following random population scheme minimizes resampling variance $\hat{V}_t^2\left[h_t\right]$:
\begin{enumerate}[label = \arabic*.]
\item
Add one particle 
$\left(w_t^{\left(N_t + 1\right)}, \xi_t^{\left(N_t + 1\right)}\right) = \left(\overline{w}_t, c\right)$
to the ensemble
$\left(w_t^{\left(j\right)}, \xi_t^{\left(j\right)}\right)_{1 \leq j \leq N_t}$.
\item
Sort the ensemble $\left(w_t^{\left(j\right)}, \xi_t^{\left(j\right)}\right)_{1 \leq j \leq N_t + 1}$ 
from highest to lowest by the value of $h_t\left(\xi_t^{\left(j\right)}\right)$
so that
\begin{equation*}
h_t\left(\xi_t^{\left(1\right)}\right) \geq h_t\left(\xi_t^{\left(2\right)}\right)
\geq \cdots \geq h_t\left(\xi_t^{\left(N_t\right)}\right) \geq h_t\left(\xi_t^{\left(N_t + 1\right)}\right)
\end{equation*}
\item
Apply stratified resampling.
\end{enumerate}
\item{\label{simplescheme}}
The fixed population scheme that minimizes 
resampling variance  $\hat{V}_t^2\left[h_t\right]$
is a simpler version of the scheme in part \ref{complexscheme}.
First sort particles from highest to lowest by the value of $h_t\left(\xi_t^{\left(i\right)}\right)$
and then apply stratified resampling.
\end{enumerate}
\end{thm}
\begin{proof}
Assume particles have been sorted so that
$h_t\left(\xi_t^{\left(1\right)}\right) \geq h_t\left(\xi_t^{\left(2\right)}\right) \geq \cdots \geq h_t\left(\xi_t^{\left(N_t\right)}\right)$
and consider an arbitrary resampling matrix $W_t$.
By Lemma \ref{simplelemma},
the resampling variance is decreased if $h_t^T W_t W_t^T h_t$ is increased.

As a first step toward increasing $h_t^T W_t W_t^T h_t$,
define $P_t$ and $N_t$ by 
\begin{equation*}
\begin{cases}
p_t^{\left(i, j\right)} = w_t^{\left(i, j\right)} \mathds{1}\left\{h_t\left(\xi_t^{\left(j\right)}\right) \geq 0\right\}, 
& 1 \leq i \leq N_t \\
q_t^{\left(i, j\right)} = w_t^{\left(i, j\right)} \mathds{1}\left\{h_t\left(\xi_t^{\left(j\right)}\right) < 0\right\}, 
& 1 \leq i \leq N_t
\end{cases}
\end{equation*}
and
\begin{equation*}
\begin{cases}
p_t^{\left(N_t + 1, j\right)} = \overline{w}_t - \sum_{k=1}^{N_t} p_t^{\left(k, j\right)} \\
q_t^{\left(N_t + 1, j\right)} = \overline{w}_t - \sum_{k=1}^{N_t} q_t^{\left(k, j\right)}
\end{cases}
\end{equation*}
Then set $S_t = \begin{pmatrix} P_t & Q_t \end{pmatrix}$  and observe that 
$h_t^T W_t W_t^T h_t  \leq h_t^T S_t S_t^T h_t$.

Let $c^{\left(1\right)}, c^{\left(2\right)}, \ldots, c^{\left(N_{t+1}\right)}$ denote the columns of $P_t$, 
sorted so that 
$h_t^T c^{\left(1\right)} \geq h_t^T c^{\left(2\right)} \geq \cdots \geq h_t^T c^{\left(N_{t+1}\right)} \geq 0$.
Consider the following algorithm to increase the value of $h_t^T P_t P_t^T h_t$:
\begin{enumerate}
\item
Call a quadruplet $\left(i,j,k,\ell\right)$
a problematic quadruplet
if $p_t^{\left(i,j+\ell\right)}>0$ and $p_t^{\left(i+k,j\right)}>0$
and if $i + k \leq N_t$.
Choose a problematic quadruplet with $i$ as small as possible.
If there is more than one such quadruplet,
choose one with $j$ as small as possible.
\item{\label{keystep}}
Set $\alpha=\min\left\{ p_t^{\left(i,j+\ell\right)},p_t^{\left(i+k,j\right)}\right\}$
and update the entries of $P_t$ with
\begin{align*}
p_t^{\left(i, j\right)} &= p_t^{\left(i, j\right)} + \alpha&
p_t^{\left(i, j+\ell\right)} &= p_t^{\left(i, j+\ell\right)} - \alpha\\
p_t^{\left(i+k, j\right)} &= p_t^{\left(i, j\right)} - \alpha&
p_t^{\left(i+k, j+\ell\right)} &= p_t^{\left(i, j+\ell\right)} + \alpha
\end{align*}
\item{\label{otherstep}}
If necessary, resort the columns 
$c^{\left(j + \ell\right)},  c^{\left(j + \ell + 1\right)}, \ldots,  c^{\left(N_{t+1}\right)}$
to ensure that $h_t^T c^{\left(j + \ell\right)} \geq h_t^T c^{\left(j + \ell + 1\right)} \geq \cdots \geq h_t^T c^{\left(N_{t+1}\right)}$.
\end{enumerate}

Note that step \ref{keystep} of the algorithm increases $h_t^T P_t^T P_t h_t$ 
or leaves $h_t^T P_t^T P_t h_t$ unchanged,
while step \ref{otherstep} ensures that
$h_t^T c^{\left(1\right)} \geq h_t^T c^{\left(2\right)} \geq \cdots \geq h_t^T c^{\left(N_{t+1}\right)} \geq 0$.
After repeated applications of the algorithm,
all the problematic quadruplets involving column
$c^{\left(1\right)}$ will eventually be gone
and the same too with columns $c^{\left(2\right)}$, $c^{\left(3\right)}$, etc.
Eventually, the algorithm will have no more problematic quadruplets to correct.
A similar algorithm can be applied to increase the value of $N_t$.
On examination it is seen that the resulting matrix $\begin{pmatrix} P_t & Q_t \end{pmatrix}$ 
generates the same resampling scheme as described in part \ref{complexscheme}.

The proof of part \ref{simplescheme} is similar.
Because fixed population schemes satisfy
$\hat{V}_t^2\left[h_t\right] = \hat{V}_t^2\left[h_t + c\right]$ for $c \in \mathbb{R}$,
it can be assumed without loss of generality 
that $h_t\left(\xi_t^{\left(i\right)}\right) > 0$ for all $\xi_t^{\left(i\right)} \neq c$.
But in this case, the schemes in part \ref{complexscheme} and part \ref{simplescheme} are identical,
and sorted stratified resampling represents the best possible strategy.
\end{proof}


The optimal scheme identified in Theorem \ref{optimal}
is an example of a \emph{sorting} scheme.
In more general sorting schemes,
particles can be sorted using any real-valued coordinate $\theta_t \colon E_t \rightarrow \mathbb{R}$
and then stratified resampling or stratified residual resampling can be used.
Sorting schemes have a long history
dating back at least to \citet{madow1944theory}.
Sorting schemes can lead to a beneficial stratification effect.
Each particle $\hat{\xi}_t^{\left(j\right)}$
is drawn from a subset of particles for which $h_t\left(\xi_t^{\left(i\right)}\right)$ values
are similar, thereby reducing resampling variance.
Theorem \ref{optimal} indicates that the best possible coordinate for sorting is $\theta_t = h_t$.
This is the first known result 
which proves that sorting particles can produce an optimal resampling scheme.

The optimal scheme in Theorem \ref{optimal} can be difficult to implement exactly,
because the function
$h_t\left(x_t\right) = \E\left[\left.\prod_{s=t+1}^{T-1} G_s f\right| X_t = x_t\right]$.
can be challenging to compute.
However, $h_t$ is not the only coordinate for sorting particles
that can lead to an effective resampling scheme.
The error formulas of the next section show that effective sorting
is possible with a wide range of different coordinates $\theta_t$.

\subsection{Asymptotic error}{\label{asymtheory}}


In past work \cite{del2004feynman, chopin2004central, douc2008limit},
a central tool for for analyzing SMC error has been 
Central Limit Theorems (CLTs) of the form
\begin{equation*}
\sqrt{N_0}\left(\frac{
\overline{w}_{T-1}}{N_0}\sum_{i=1}^{N_T}
f\left(\hat{\xi}_{T-1}^{\left(i\right)},\xi_{T}^{\left(i\right)}\right)
- \E\left[\prod_{t=0}^{T-1} G_t f\right]\right)
\stackrel{\mathcal{D}}{\rightarrow} N\left(0, \eta^2\right)
\end{equation*}
where the quantity $\eta^2$ depends on the particular resampling scheme that is used.
CLTs have been proved for multinomial, multinomial residual and Bernoulli resampling
\cite{del2004feynman, chopin2004central, douc2008limit}.
In the present section, new error formulas are presented for stratified resampling and stratified residual resampling.
These error formulas are not CLTs;
instead they are upper bounds on \emph{asymptotic error}.
Asymptotic error is a new way to measure error that is more general than a CLT
and also more flexible for analysis.
Before presenting asymptotic error formulas,
it is therefore necessary to introduce the key features
of asymptotic error and explain how this error measurement tool can be interpreted.

Asymptotic error is a far-reaching generalization of the error rate in a CLT.
In a CLT, a sequence of random variables $\left(Y_n\right)_{n \geq 1}$ 
approach a constant $c$ with error measured by an error rate $U_n$.
\begin{equation*}
\frac{Y_n - c}{U_n} \stackrel{\mathcal{D}}{\rightarrow} N\left(0, 1\right)
\end{equation*}
Thus, a CLT can only be proved when there is very precise knowledge of the error rate $U_n$.
In contrast, asymptotic error can be analyzed when knowledge of $U_n$ is less precise
and there is only an upper or lower bound on $U_n$.
A full definition of asymptotic error is provided below:
\begin{definition}{\label{def:asymerror}}
Suppose random variables $\left(Y_n\right)_{n \geq 1}$ satisfy
\begin{equation*}
\label{geq}
\liminf_{n \rightarrow \infty} \E\left[\mathds{1}_{A_n} \left|\frac{Y_n - c}{U_n}\right|^2\right] \geq 1
\end{equation*}
for all possible sequences of sets $\left(A_n\right)_{n \geq 1}$ with $\Prob\left(A_n\right) \rightarrow 1$.
Then, $Y_n$ converges to $c$ with asymptotic error greater than or equal to $U_n$,
and write $\left|Y_n - c\right| \gtrsim U_n$.

Suppose random variables $\left(Y_n\right)_{n \geq 1}$ satisfy
\begin{equation*}
\label{leq}
\limsup_{n \rightarrow \infty} \E\left[\mathds{1}_{B_n} \left|\frac{Y_n - c}{U_n}\right|^2\right] \leq 1
\end{equation*}
for some sequence of sets $\left(B_n\right)_{n \geq 1}$ with $\Prob\left(B_n\right) \rightarrow 1$.
Then, $Y_n$ converges to $c$ with asymptotic error less than or equal to
$U_n$, and write $\left|Y_n - c\right| \lesssim U_n$.

If both conditions are satisfied, $Y_n$ converges to $c$ with asymptotic error $U_n$, and write $\left|Y_n - c\right| \sim U_n$.
\end{definition}

A CLT can be viewed as a particular example
of asymptotic error,
as guaranteed by the following lemma:

\begin{lem}{\label{guarantee}}
Suppose random variables $\left(Y_n\right)_{n \geq 1}$ satisfy
$\frac{Y_n - c}{U_n} \stackrel{\mathcal{D}}{\rightarrow} N\left(0, 1\right)$
as $n \rightarrow \infty$.
Then, $\left|Y_n - c\right| \sim U_n$.
\end{lem}
\begin{proof}
Fatou's Lemma shows
$\liminf_{n \rightarrow \infty} \E\left[\mathds{1}_{A_n} \left|\frac{Y_n - c}{U_n}\right|^2\right] \geq 1$ 
for all sequences $\left(A_n\right)_{n \geq 1}$ with $\Prob\left(A_n\right) \rightarrow 1$.
Thus, $\left|Y_n - c\right| \gtrsim U_n$.
To show $\left|Y_n - c\right| \lesssim U_n$, construct a sequence $\left(B_n\right)_{n \geq 1}$ 
with the properties $\Prob\left(B_n\right) \rightarrow 1$ and $limsup_{n \rightarrow \infty} \E\left[\mathds{1}_{B_n}  \left|\frac{Y_n-c}{U_n}\right|^2\right] \leq 1$.
First let $L_n$ be the largest number such that 
\begin{equation*}
\E\left[\mathds{1}\left\{\left|\frac{Y_n - c}{U_n}\right| < L_n\right\} \left|\frac{Y_n-c}{U_n}\right|^2\right] \leq 1
\end{equation*}
and note that $L_n$ is well-defined by the Monotone Convergence Theorem.
Set $B_n = \left\{\frac{\left|Y_n - c\right|}{U_n} < L_n\right\}$.
For any $\epsilon > 0$,
choose $M > 0$ large enough that 
$\Prob\left\{\left|Z\right| < M\right\} > 1 - \frac{\epsilon}{2}$ where $Z \sim N\left(0, 1\right)$.
Since $\frac{Y_n - c}{U_n} \stackrel{\mathcal{D}}{\rightarrow} N\left(0, 1\right)$,
it follows that
$\Prob\left\{\frac{\left|Y_n - c\right|}{U_n} < M\right\} > 1 - \epsilon$
for all $n$ large enough.
Since $x \mapsto \mathds{1}\left\{\left|x\right| < M\right\}\left|x\right|^2$
is bounded and piecewise continuous, 
\begin{equation*}
\E\left[\mathds{1}\left\{\left|\frac{Y_n - c}{U_n}\right| < M\right\} \left|\frac{Y_n-c}{U_n}\right|^2\right]
\rightarrow \E\left[\mathds{1}\left\{\left|Z\right| < M\right\} Z^2 \right] < 1
\end{equation*}
For all $n$ large enough, it follows that
$M < L_n$, $\left\{\frac{\left|Y_n - c\right|}{U_n} < M\right\} \subseteq B_n$, and 
$\Prob \left( B_n \right) \geq \Prob\left\{\frac{\left|Y_n - c\right|}{U_n} < M\right\} > 1 - \epsilon$.
Since $\epsilon$ is arbitrary, $\Prob\left(B_n\right) \rightarrow 1$.
\end{proof}

Asymptotic error can be compared to mean squared error,
which is another common error metric, different from the error rate in the CLT.
Both asymptotic error and mean squared error are tools
to assess the value of an estimate and to provide confidence intervals around an estimate.
By Chebyshev's inequality, asymptotic error leads to confidence intervals:
\begin{equation*}
\left|Y_n - c\right| \lesssim U_n \implies
\limsup_{n \rightarrow \infty} \Prob\left\{\left|Y_n - c\right| \geq \epsilon U_n \right\} \leq \frac{1}{\epsilon^2}
\end{equation*}
The chief difference between asymptotic error and mean squared error
is robustness to perturbations.
Mean squared error $\E\left|Y_n - c\right|^2$ is quite sensitive to changes in 
the behavior of $Y_n$ on a set of vanishing probability,
but asymptotic error is completely robust to these changes.
Thus the confidence intervals derived from asymptotic error bounds
can be much tighter than those derived from mean squared error bounds.

The rigorous treatment of asymptotic error 
leads to new results in SMC analysis, 
including the first known error formulas for 
stratified resampling and stratified residual resampling.
In the following theorem,
these new formulas are presented alongside CLTs
for multinomial, multinomial residual, and Bernoulli resampling, 
which are extended
from \cite{del2004feynman, chopin2004central, douc2008limit} to 
have less restrictions on functions $\left(G_t\right)_{t \geq 0}$ and $f$:

\begin{thm}{\label{multvar}}
Assume $\E\left[\prod_{s=0}^t G_s\right] < \infty$ for $0 \leq t \leq T-1$, $\E\left|G_0 h_0\right|^2 < \infty$,
and $\E\left[\prod_{s=0}^t G_s \left|G_{t+1} h_{t+1}\right|^2\right]$ for $0 \leq t \leq T-1$.
If multinomial residual or stratified residual resampling is used, assume $\E\left[\prod_{s=0}^{t-1} G_s \mathds{1}\left\{\tilde{G}_t \in \left\{1,2,\ldots\right\} \right\}\right] = 0$ for $0 \leq t \leq T-1$ as well.
Set
$\tilde{G}_t = \E\left[\prod_{s=0}^{t-1} G_s \right] G_t \slash \E\left[\prod_{s=0}^t G_s\right]$
and set
\begin{equation*}
\eta^2 = \Var\left[G_0 h_0\right] + \sum_{t=0}^{T-1} \hat{\eta}_t^2 \left[h_t\right]
+ \sum_{t=0}^{T-1} \E\left[\prod_{s=0}^t G_s\right] 
\E\left[\prod_{s=0}^t G_s \Var\left[\left.G_{t+1} h_{t+1} \right| X_t\right]\right]
\end{equation*}
where $\eta^2$ depends on a sequence of a numbers $\left(\hat{\eta}_t\left[h_t\right]^2\right)_{0 \leq t \leq T-1}$.

First assume multinomial resampling, Bernoulli resampling, or multinomial residual resampling is used.
Then SMC estimates satisfy the CLT
\begin{equation*}
\sqrt{N_0}\left(\frac{
\overline{w}_{T-1}}{N_0}\sum_{i=1}^{N_T}
f\left(\hat{\xi}_{T-1}^{\left(i\right)},\xi_{T}^{\left(i\right)}\right)
- \E\left[\prod_{t=0}^{T-1} G_t f\right]\right)
\stackrel{\mathcal{D}}{\rightarrow} N\left(0, \eta^2\right)
\end{equation*}
where $\hat{\eta}_t^2\left[h_t\right]$ is determined by the resampling scheme:

\begin{flalign*}
\begin{array}{l l}
\text{multinomial} & 
\left(\E\left[\prod_{s=0}^t G_s\right]\right)^2
\min_{c \in \mathbb{R}}
\E\left[\prod_{s=0}^t \tilde{G_s} 
\left|h_t - c \right|^2\right] 
\\[.3cm]
\text{multinomial residual} &
\left(\E\left[\prod_{s=0}^t G_s\right]\right)^2
\min_{c \in \mathbb{R}}
\E\left[\prod_{s=0}^{t-1} \tilde{G}_s \left\{\tilde{G}_t\right\} \left|h_t - c\right|^2 \right]
\\[.3cm]
\text{Bernoulli} &
\left(\E\left[\prod_{s=0}^t G_s\right]\right)^2
\E\left[\prod_{s=0}^{t-1} \tilde{G}_s \left\{\tilde{G}_t\right\}\left(1 - \left\{\tilde{G}_t\right\}\right) 
h_t^2\right] \end{array}
&&
\end{flalign*}

Next, assume that at each resampling step particles are sorted by a coordinate $\theta_t$
and then stratified or stratified residual resampling is used.
Then, 
\begin{equation*}
\left|\frac{\overline{w}_{T-1}}{N_0}\sum_{i=1}^{N_T}
f\left(\hat{\xi}_{T-1}^{\left(i\right)},\xi_{T}^{\left(i\right)}\right)
- \E\left[\prod_{t=0}^{T-1} G_t f\right]\right|
\lesssim \frac{\eta}{\sqrt{N_0}}
\end{equation*}
where $\hat{\eta}_t^2\left[h_t\right]$ is determined by the resampling scheme:
\begin{flalign*}
\begin{array}{l l}
\text{stratified} &
\left(\E\left[\prod_{s=0}^t G_s\right]\right)^2
\min_{p\colon \mathbb{R} \rightarrow \mathbb{R}}
\E \left[\prod_{s=0}^t \tilde{G}_s \left|h_t - p \left(\theta_t\right)\right|^2\right]
\\[.3cm]
\text{stratified residual} &
\left(\E\left[\prod_{s=0}^t G_s\right]\right)^2
\min_{p\colon \mathbb{R} \rightarrow \mathbb{R}}
\E \left[\prod_{s=0}^{t - 1} \tilde{G}_s \left\{\tilde{G}_t\right\} \left|h_t - p \left(\theta_t\right)\right|^2\right]
\end{array} &&
\end{flalign*}
\end{thm}

%
There are two main conclusions that can be drawn from Theorem \ref{multvar}
about how best to choose a resampling scheme.
The first conclusion
is that residual versions of a resampling scheme should be used whenever possible.
Error formulas for multinomial and multinomial residual
resampling are differentiated by a factor of $\tilde{G}_t$ for multinomial 
and a factor of 
$\left\{\tilde{G}_t\right\}$ for multinomial residual resampling.
Since $\left\{\tilde{G}_t\right\}$ is always as low or lower than $\tilde{G}_t$,
the multinomial residual resampling scheme can lead to reduced SMC error.
Similarly, stratified residual resampling has an improved asymptotic error upper bound
compared to stratified resampling.

The second conclusion that follows from Theorem \ref{multvar}
is that sorting schemes can substantially reduce error,
depending on the coordinate $\theta_t$ used for sorting.
Error formulas for multinomial and stratified resampling are distinguished by 
a factor of $\min_{c \in \mathbb{R}}
\E\left[\prod_{s=0}^t \tilde{G_s} 
\left|h_t - c \right|^2\right]$
for multinomial and a factor of
$\min_{p \colon \mathbb{R} \rightarrow \mathbb{R}}
\E \left[\prod_{s=0}^t \tilde{G}_s \left|h_t - p \left(\theta_t\right)\right|^2\right]$
for stratified resampling.
Since
$\min_{p \colon \mathbb{R} \rightarrow \mathbb{R}}
\E \left[\prod_{s=0}^t \tilde{G}_s \left|h_t - p \left(\theta_t\right)\right|^2\right]
\leq \min_{c \in \mathbb{R}}
\E\left[\prod_{s=0}^t \tilde{G_s} 
\left|h_t - c \right|^2\right]$,
asymptotic error for stratified resampling is as low or lower
than asymptotic error for multinomial resampling.
In the simplest case where $\theta_t \equiv 0$,
particles are not sorted in any particular order and error reduction may be very mild;
on the other hand,
as the stratification effect due to sorting by $\theta_t$ increases,
the error contributed at each resampling step
approaches zero.
Similarly, asymptotic error for stratified residual resampling
is as low or lower than asymptotic error for multinomial residual resampling,
with a major reduction possible
depending on the coordinate $\theta_t$.

Below, two examples of resampling schemes that use sorting to achieve error reduction are described:

\begin{example}[Sorting in $\mathbb{R}^d$]
When applying SMC to a one-dimensional system,
\citet{kitagawa1996monte} sorted particles $\xi_t^{\left(i\right)}$ by their values in $\mathbb{R}$ and 
then applied stratified resampling, 
leading to a dramatic reduction in resampling variance.
Later, \citet{gerber2017negative} suggested a more general strategy
of sorting particles in $\mathbb{R}^d$
according to a Hilbert curve, a measurable one-to-one mapping
from $\mathbb{R}^d$ into $\mathbb{R}$.
In both cases, Theorem \ref{multvar} gives an upper bound on asymptotic error 
with $\hat{\eta}_t^2\left[h_t\right] = 0$.
This is the lowest possible asymptotic error for any SMC scheme.
It should be noted however that 
pre-asymptotic resampling variance for this sorting strategy
is difficult to estimate; further research may help
elucidate the practical efficiency of Hilbert curve sorting.
\end{example}
\begin{example}[Binning]
In binned resampling \cite{huber1996weighted},
the state space is sorted into bins $B_1, B_2, \ldots, B_K$,
and particles $\xi_t^{\left(i\right)}$ are arranged 
by bin number, from highest to lowest.
When stratified resampling is applied,
Theorem \ref{multvar} gives an upper bound on asymptotic error with
\begin{equation*}
\label{binformula}
\hat{\eta}_t^2\left[h_t\right] = 
\left(\E \left[\prod_{s=0}^t G_s\right]\right)^2
\E \left[\prod_{s=0}^t \tilde{G}_s \sum_{k=1}^K \mathds{1}_{B_k}
\left|h_t - \frac{\E \left[\prod_{s=0}^t \tilde{G}_s \mathds{1}_{B_k} h_t\right]}
{\E \left[\prod_{s=0}^t \tilde{G}_s\right]}\right|^2\right]
\end{equation*}
As values of $h_t$ become increasingly similar in each bin $B_k$,
equation \eqref{binformula} guarantees that asymptotic error must decrease.
In particular, as the diameter of the bins approaches zero in a region that grows to fill the state space $E_t$,
$\hat{\eta}_t^2\left[h_t\right]$ approaches the lowest possible level: $\hat{\eta}_t^2\left[h_t\right] = 0$.
\end{example}
 
 

\section{Conclusion} \label{sec:Conclusion}

The present work derives a theoretical framework that unifies past SMC scholarship
and establishes significant new results.
The framework uses a simple parametrization to describe a great variety of resampling schemes.
The theoretical framework includes a unified error analysis and asymptotic error formulas with a unified structure
that can be used to compare resampling schemes.

The resampling matrix framework combines a fresh look at common resampling schemes
with new technical tools to analyze SMC error.
Aymptotic error is defined in a new way,
as mean squared error outside a set of vanishing probability.
This notion of error leads to simple proofs and rigorous comparisons between resampling schemes.
Due to this innovation, asymptotic error formulas are now available
for stratified resampling and stratified residual resampling,
including the full range of unbounded functions $\left(G_t\right)_{t \geq 0}$ and $f$
used in practical implementations of SMC.


The framework leads to two concrete recommendations for how best to resample:
\begin{enumerate}
\item
Firstly, practitioners are encouraged to use stratified residual resampling instead of
multinomial residual resampling and stratified resampling instead of multinomial resampling
in order to reduce resampling variance.
Similar recommendations were given in \citet{douc2005comparison},
but resampling matrices provide a more intuitive and general explanation 
for reductions in resampling variance.
\item
Secondly, sorting schemes can lead to extremely low asymptotic error rates.
These schemes are recommended when there is a coordinate $\theta_t$
that can be used to sort particles $\left(\xi_t^{\left(i\right)}\right)_{1 \leq i \leq N_t}$
to achieve a beneficial stratification effect in the resampling step.
\end{enumerate}

In summary, the unifying analysis in the current paper
shines light on the best ways to resample,
providing practical guidance to help SMC users make the most of
the powerful and versatile SMC algorithm.



\section{Appendix}

\subsection{Estimates of ratios}

SMC is often used to approximate ratios
\begin{equation*}
\frac{\sum_{j=1}^{N_{T}}\hat{w}_{T-1}^{\left(j\right)}f\left(\hat{\xi}_{T-1}^{\left(j\right)},\xi_{T}^{\left(j\right)}\right)}
{\sum_{j=1}^{N_T}\hat{w}_{T-1}^{\left(j\right)}
\mathds{1}\left\{\hat{\xi}_{T-1}^{\left(j\right)} \neq c\right\}}
\approx\frac{\E\left[\prod_{t=0}^{T-1} G_t f \right]}{\E\left[\prod_{t=0}^{T-1} G_t \right]}
\end{equation*}
In some cases, the denominator in the SMC estimate may equal zero,
and the estimate can be assigned 
an arbitrary value when this occurs.
If $\E\left[\prod_{s=0}^{t} G_s \right]<\infty$ for $0 \leq t \leq T-1$
and $\E\left|\prod_{t=0}^{T-1} G_t f \right|<\infty$, then
Theorem \ref{weak} guarantees
\begin{equation*}
\begin{cases}
\sum_{j=1}^{N_{T}}\hat{w}_{T-1}^{\left(j\right)}f\left(\hat{\xi}_{T-1}^{\left(j\right)},\xi_{T}^{\left(j\right)}\right) 
\stackrel{\Prob}{\rightarrow}
\E\left[\prod_{t=0}^{T-1} G_t f \right]
\\
\sum_{j=1}^{N_T}\hat{w}_{T-1}^{\left(j\right)}
\mathds{1}\left\{\hat{\xi}_{T-1}^{\left(j\right)} \neq c\right\}
\stackrel{\Prob}{\rightarrow}
\E\left[\prod_{t=0}^{T-1} G_t \right]
\end{cases}
\end{equation*}
Therefore, SMC estimates of ratios are convergent.
While expressions for the bias and variance of these estimates are challenging to derive, 
asymptotic error for these estimates can be studied with the aid of the following lemma:

\begin{lem}{\label{slutskylemma}}
\begin{enumerate}
\item{\label{slutsky}}
If $Y_n \stackrel{\mathcal{D}}{\rightarrow} N\left(0, 1\right)$ 
and $W_n \stackrel{\Prob}{\rightarrow} c \in \mathbb{R}$, 
then $Y_n W_n \stackrel{\mathcal{D}}{\rightarrow} N\left(0, c^2\right)$.
\item{\label{notslutsky}}
If $\left|Y_n\right| \lesssim 1$ and $W_n \stackrel{\Prob}{\rightarrow} c > 0$,
then $\left|Y_n W_n\right| \lesssim c$.
\end{enumerate}
\end{lem}
\begin{proof}
Part \ref{slutsky} follows from Slutsky's Theorem.
To prove part \ref{notslutsky}, first construct a sequence of sets $\left(C_n\right)_{n \geq 1}$ with the properties
$\Prob\left(C_n\right) \rightarrow 1$ and 
$\limsup_{n \rightarrow \infty} \left\lVert \mathds{1}_{C_n} \frac{W_n}{c} \right\rVert_{\infty} \leq 1$.
Set $E_{m,n} = \left\{\frac{W_n}{c} \leq 1 + \frac{1}{m}\right\}$ for $m,n \geq 1$
and $D_{1,n} = E_{1,n}$ for $n \geq 1$.
By the hypothesis in part \ref{notslutsky},
there exists a number $N\left(m\right) \geq 1$ such that $\Prob\left(E_{m,n}\right) \geq 1 - \frac{1}{m}$
for $n \geq N\left(m\right)$.
Accordingly, for $m > 1$ define
\begin{equation*}
\begin{cases}
D_{m,n} = D_{m-1,n}, & n < N\left(m\right) \\
D_{m,n} = E_{m,n}, & n \geq N\left(m\right)
\end{cases}
\end{equation*}
By this construction,
for $m \geq M$ and $n \geq N\left(M\right)$,
$\left\lVert \mathds{1}_{D_{m,n}} \frac{W_n}{c}\right\rVert_{\infty} \leq 1 + \frac{1}{M}$
and $\Prob\left(D_{m,n}\right) \geq 1 - \frac{1}{M}$.
Setting $C_n = D_{n,n}$ gives the required sequence.
Lastly, select $\left(B_n\right)_{n \geq 1}$ so that $\Prob\left(B_n\right) \rightarrow 1$ and
$\limsup_{n \rightarrow \infty} \E\left[\mathds{1}_{B_n} Y_n^2\right] \leq 1$.
Then $D_n = B_n \cap C_n$ satisfies $\Prob\left(D_n\right) \rightarrow 1$ and
$\limsup_{n \rightarrow \infty} \E\left[\mathds{1}_{D_n} \left|\frac{Y_n W_n}{c}\right|^2\right]
\leq 1$.
\end{proof}

To apply Lemma \ref{slutskylemma}, set $\tilde{f} = f - \frac{\E\left[\prod_{t=0}^{T-1} G_t f\right]}{\E\left[\prod_{t=0}^{T-1} G_t \right]}$ and observe
\begin{equation*}
\frac{\sum_{j=1}^{N_T} \hat{w}_{T-1}^{\left(j\right)} \tilde{f}\left(\hat{\xi}_{T-1}^{\left(j\right)},\xi_{T}^{\left(j\right)}\right)}
{ \sum_{j=1}^{N_T} \hat{w}_{T-1}^{\left(j\right)} \mathds{1}\left\{\hat{\xi}_{T-1} \neq c\right\}}
= \frac{\sum_{j=1}^{N_{T}}\hat{w}_{T-1}^{\left(j\right)}f\left(\hat{\xi}_{T-1}^{\left(j\right)},\xi_{T}^{\left(j\right)}\right)}
{\sum_{j=1}^{N_T}\hat{w}_{T-1}^{\left(j\right)}
\mathds{1}\left\{\hat{\xi}_{T-1}^{\left(j\right)} \neq c\right\}} 
- \frac{\E\left[\prod_{t=0}^{T-1} G_t f \right]}{\E\left[\prod_{t=0}^{T-1} G_t \right]}
\end{equation*}
Since $\frac{1}{N_0}  \sum_{j=1}^{N_T} \hat{w}_{T-1}^{\left(j\right)} \mathds{1}\left\{\hat{\xi}_{T-1} \neq c\right\}
\stackrel{\Prob}{\rightarrow} \E\left[\prod_{t=0}^{T-1} G_t\right]$,
the asymptotic error of an SMC estimate is the asymptotic error of 
$ \frac{1}{N_0} \sum_{j=1}^{N_T} \hat{w}_{T-1}^{\left(j\right)} \tilde{f}\left(\hat{\xi}_{T-1}^{\left(j\right)},\xi_{T}^{\left(j\right)}\right)$
scaled by a factor of $\left(\E\left[\prod_{t=0}^{T-1} G_t\right]\right)^{-1}$.
A corollary of Theorem \ref{multvar} gives precise expressions for asymptotic error:

\begin{cor}
Set
$\tilde{G}_t = \E\left[\prod_{s=0}^{t-1} G_s \right] G_t \slash \E\left[\prod_{s=0}^t G_s\right]$
and set
\begin{equation*}
\tilde{h}_t \left(x_t\right) = \E\left[\left.\prod_{s=t+1}^{T-1} \tilde{G_s} 
\left(f - \frac{\E\left[\prod_{r=0}^{T-1} G_r f\right]}{\E\left[\prod_{r=0}^{T-1} G_r\right]}\right)\right| X_t = x_t\right]
\end{equation*}
Assume that $\E\left[\prod_{s=0}^t G_s\right] < \infty$ for $0 \leq t \leq T-1$, $\E\left|G_0 \tilde{h}_0\right|^2 < \infty$,
and $\E\left[\prod_{s=0}^t G_s \left|G_{t+1} \tilde{h}_{t+1}\right|^2\right]$ for $0 \leq t \leq T-1$.
If multinomial residual or stratified residual resampling is used, also assume $\E\left[\prod_{s=0}^{t-1} G_s \mathds{1}\left\{\tilde{G}_t \in \left\{1,2,\ldots\right\} \right\}\right] = 0$ for $0 \leq t \leq T-1$.
Define
\begin{equation*}
\eta^2 = \Var\left[\tilde{G}_0 \tilde{h}_0\right] + \sum_{t=0}^{T-1} \hat{\eta}_t^2 \left[\tilde{h}_t\right]
+ \sum_{t=0}^{T-1} 
\E\left[\prod_{s=0}^t \tilde{G}_s \Var\left[\left.\tilde{G}_{t+1} \tilde{h}_{t+1} \right| X_t\right]\right]
\end{equation*}
where $\eta^2$ depends on a sequence of numbers $\left(\hat{\eta}_t\left[\tilde{h}_t\right]^2\right)_{0 \leq 1 \leq T-1}$.

First assume multinomial resampling, Bernoulli resampling, or multinomial residual resampling is used.
Then SMC estimates satisfy the CLT
\begin{equation*}
\sqrt{N_0}\left(\frac{\sum_{j=1}^{N_{T}}\hat{w}_{T-1}^{\left(j\right)}f\left(\hat{\xi}_{T-1}^{\left(j\right)},\xi_{T}^{\left(j\right)}\right)}
{\sum_{j=1}^{N_T}\hat{w}_{T-1}^{\left(j\right)}
\mathds{1}\left\{\hat{\xi}_{T-1}^{\left(j\right)} \neq c\right\}} 
- \frac{\E\left[\prod_{t=0}^{T-1} G_t f \right]}{\E\left[\prod_{t=0}^{T-1} G_t \right]}\right)
\stackrel{\mathcal{D}}{\rightarrow} N\left(0, \eta^2\right)
\end{equation*}
where $\hat{\eta}_t^2\left[\tilde{h}_t\right]$ is determined by the resampling scheme:

\begin{align*}
\begin{array}{l l}
\text{multinomial} & 
\E\left[\prod_{s=0}^t \tilde{G_s} 
\tilde{h}_t^2\right]
\\[.3cm]
\text{multinomial residual} &
\min_{c \in \mathbb{R}}
\E\left[\prod_{s=0}^{t-1} \tilde{G}_s \left\{\tilde{G}_t\right\} \left|\tilde{h}_t - c\right|^2 \right]
\\[.3cm]
\text{Bernoulli} &
\E\left[\prod_{s=0}^{t-1} \tilde{G}_s \left\{\tilde{G}_t\right\}\left(1 - \left\{\tilde{G}_t\right\}\right) 
\tilde{h}_t^2\right]
\end{array}
\end{align*}

Next assume at each resampling step particles are sorted by a coordinate $\theta_t$
and then stratified or stratified residual resampling is used.
Then, 
\begin{equation*}
\left|\frac{\sum_{j=1}^{N_{T}}\hat{w}_{T-1}^{\left(j\right)}f\left(\hat{\xi}_{T-1}^{\left(j\right)},\xi_{T}^{\left(j\right)}\right)}
{\sum_{j=1}^{N_T}\hat{w}_{T-1}^{\left(j\right)}
\mathds{1}\left\{\hat{\xi}_{T-1}^{\left(j\right)} \neq c\right\}} 
- \frac{\E\left[\prod_{t=0}^{T-1} G_t f \right]}{\E\left[\prod_{t=0}^{T-1} G_t \right]}\right|
\lesssim \frac{\eta}{\sqrt{N_0}}
\end{equation*}
where $\hat{\eta}_t^2\left[\tilde{h}_t\right]$ is determined by the resampling scheme:
\begin{align*}
\begin{array}{l l}
\text{stratified} &
\min_{p\colon \mathbb{R} \rightarrow \mathbb{R}}
\E \left[\prod_{s=0}^t \tilde{G}_s \left|\tilde{h}_t - p \left(\theta_t\right)\right|^2\right]
\\[.3cm]
\text{stratified residual} &
\min_{p\colon \mathbb{R} \rightarrow \mathbb{R}}
\E \left[\prod_{s=0}^{t - 1} \tilde{G}_s \left\{\tilde{G}_t\right\} \left|\tilde{h}_t - p \left(\theta_t\right)\right|^2\right]
\end{array}
\end{align*}
\end{cor}

\subsection{Proofs for Theorems \ref{weak}, \ref{bounded}, and \ref{multvar}}

To prove Theorem \ref{weak}, first introduce intermediate $\sigma$-algebras between $\mathcal{F}_{t-1}$
and $\mathcal{F}_t$:
\begin{equation*}
\begin{cases}
\mathcal{F}_t^{\left(0\right)} = \mathcal{F}_{t-1} \\
\mathcal{F}_t^{\left(i\right)} = \mathcal{F}_t^{\left(i-1\right)} \vee \sigma\left(\xi_t^{\left(i\right)}\right),
& t = 0, 1 \leq i \leq N_t - 1 \\
\mathcal{F}_t^{\left(i\right)} = \mathcal{F}_t^{\left(i-1\right)} \vee \sigma\left(\hat{\xi}_{t-1}^{\left(i\right)}, \xi_t^{\left(i\right)}\right),
& t > 0, 1 \leq i \leq N_t - 1 \\
\mathcal{F}_t^{\left(N_t\right)} = \mathcal{F}_t
\end{cases}
\end{equation*}
Next, define a martingale $M_t^{\left(i\right)} = \E\left[\left.\frac{1}{N_0}\sum_{k=1}^{N_T} \hat{w}_{T-1}^{\left(k\right)}
f\left(\hat{\xi}_{T-1}^{\left(k\right)}, \xi_T^{\left(k\right)}\right)\right| \mathcal{F}_t^{\left(i\right)}\right]$.
Since pairs
$\left(w_t^{\left(k\right)}, \xi_t^{\left(k\right)}\right)$ are conditionally independent
given $\mathcal{F}_{t-1}$,
it follows
\begin{equation}{\label{martrep2}}
M_t^{\left(i\right)} = \frac{1}{N_0} \sum_{k=1}^i w_t^{\left(k\right)} h_t \left(\xi_t^{\left(k\right)}\right)
 + \frac{1}{N_0} \sum_{k=i+1}^{N_t} \E\left[\left. w_t^{\left(k\right)} h_t \left(\xi_t^{\left(k\right)}\right)
 \right| \mathcal{F}_{t-1}\right]
\end{equation}
The proof of Theorem \ref{weak} also requires two technical lemmas.
\begin{lem} \label{condconv}
For each $n\geq1$, suppose $\mathcal{G}_{n0}\subseteq\mathcal{G}_{n1}\subseteq\mathcal{G}_{n2}\subseteq\cdots\subseteq\mathcal{G}_{n,k_n}$
is a filtration and $\left(Y_{nj}\right)_{1\leq j\leq k_n}$
is a sequence of random variables with $Y_{nj}$ measurable in $\mathcal{G}_{nj}$.
Suppose
\begin{equation*}
\label{twoconditions}
\begin{cases}
\sum_{j=1}^{k_n} \E\left[\left|Y_{nj}\right|\mathds{1}\left\{ \left|Y_{nj}\right|>C\right\} \rvert \mathcal{G}_{n,j-1}\right]\stackrel{\Prob}{\rightarrow} 0, & C > 0 \\
\lim_{\lambda \rightarrow \infty} \sup_{n\geq1} \Prob \left\{ \sum_{j=1}^{k_n} \E\left[\left|Y_{nj}\right|\rvert\mathcal{G}_{n,j-1}\right]>\lambda\right\} = 0
\end{cases}
\end{equation*}
Then, $\sum_{j=1}^{k_n}\left\{ Y_{nj}-\E\left[Y_{nj}\rvert\mathcal{G}_{n,j-1}\right]\right\} \stackrel{P}{\rightarrow}0$.
\end{lem}
\begin{proof}
The lemma can be traced back to the early martingale literature,
particularly \citet[pg. 45-47]{hall1980martingale} and \citet[pg. 626]{mcleish1974dependent}.
The lemma appears in \citet{douc2008limit},
who also use the lemma to prove convergence of SMC schemes.
\end{proof}

\begin{lem}{\label{weaklemma}}
If $\E\left[\prod_{s=0}^{t} G_s \right]<\infty$ for $0 \leq t \leq T-1$
and $\E\left|\prod_{t=0}^{T-1} G_t f \right|<\infty$, then for each $0 \leq t \leq T$
and $C > 0$
\begin{align}
\label{weaklindeberg}
& 
\E\left[\left. \frac{1}{N_0} \sum_{i=1}^{N_t} \hat{w}_{t-1}^{\left(i\right)} 
\left| f\left(\xi_t^{\left(i\right)}\right) \right|
\mathds{1}\left\{\frac{1}{N_0} \hat{w}_{t-1}^{\left(i\right)} \left|f\left(\xi_t^{\left(i\right)}\right)\right|\geq C\right\} 
\right\rvert\mathcal{F}_{t-1} \right]
\stackrel{\Prob}{\rightarrow} 0
\end{align}
\end{lem}
\begin{proof}[Proof of Lemma \ref{weaklemma}]
Use induction on the time index $0 \leq t \leq T$.
For the $t = 0$ case, the Dominated Convergence Theorem shows
\begin{align*}
& \quad \E\left[\frac{1}{N_0}\sum_{i=1}^{N_0} 
G_0\left(\xi_0^{\left(i\right)}\right)
\left|f\left(\xi_0^{\left(i\right)}\right)\right|
\mathds{1}\left\{\frac{1}{N_0}G_0\left(\xi_0^{\left(i\right)}\right)
\left|f\left(\xi_0^{\left(i\right)}\right)\right|\geq C\right\}\right] \\
&= \E\left[G_0 \left|f\right| \mathds{1}\left\{ \frac{1}{N_0} G_0 \left|f\right| \geq C\right\}\right]
\rightarrow 0
\end{align*}

Next, assume \eqref{weaklindeberg} holds for all times $0 \leq s \leq t-1$ and consider a time $t \geq 1$.
By the induction assumption,
\begin{align*}
& \E\left[\left.\frac{1}{N_0} \sum_{i=1}^{N_{t-1}}
w_{t-1}^{\left(i\right)}
\mathds{1}\left\{\frac{1}{N_0} w_{t-1}^{\left(i\right)} \geq C\right\} 
\right\rvert\mathcal{F}_{t-1}\right]
\stackrel{\Prob}{\rightarrow} 0, 
& C>0
\end{align*}
For $\delta > 0$, calculate
\begin{align*}
& \quad \Prob\left\{\max_{1 \leq i \leq N_{t-1}} \frac{w_{t-1}^{\left(i\right)}}{N_0} \geq \delta\right\}
= \E\left[\Prob\left\{\left.\max_{1 \leq i \leq N_{t-1}} \frac{w_{t-1}^{\left(i\right)}}{N_0} \geq \delta 
\right|\mathcal{F}_{t-1}\right\}\right] \\
& \leq \E\left[\min\left\{1, 
\frac{1}{\delta} \E\left[\left.\frac{1}{N_0} \sum_{i=1}^{N_{t-1}}
w_{t-1}^{\left(i\right)}
\mathds{1}\left\{\frac{1}{N_0} w_{t-1}^{\left(i\right)} \geq \delta\right\} 
\right\rvert\mathcal{F}_{t-1}\right]\right\}\right]
\end{align*}
Sending $N_0$ to infinity, it follows that
$\max_{1 \leq i \leq N_{t-1}} \frac{w_{t-1}^{\left(i\right)}}{N_0} 
\stackrel{\Prob}{\rightarrow} 0$, and by Assumption \ref{assumption1} 
$\max_{1\leq j \leq N_t}\frac{\hat{w}_{t-1}^{\left(j\right)}}{N_0}\stackrel{P}{\rightarrow}0$.

Now for $C>0$ and $\delta>0$, define
$\nu = \left|f\right| \mathds{1}\left\{\left|f\right| \geq \frac{C}{\delta}\right\}$.
Calculate
\begin{align*}
& \quad \Prob \left\{
\E\left[\left.\frac{1}{N_{0}} \sum_{i=1}^{N_t}
\hat{w}_{t-1}^{\left(i\right)} \left| f\left(\xi_t^{\left(i\right)}\right) \right|
\mathds{1}\left\{\frac{1}{N_0} \hat{w}_{t-1}^{\left(i\right)}\left| f\left(\xi_t^{\left(i\right)}\right) \right|\geq C\right\} 
\right\rvert\mathcal{F}_t\right]
> \epsilon\right\} \\
& \leq 
\Prob \left\{\max_{1\leq j\leq N_t}\frac{\hat{w}_{t-1}^{\left(j\right)}}{N_0} > \delta\right\}
+ \frac{1}{\epsilon}
\E\left[\frac{1}{N_{0}} \sum_{i=1}^{N_t} \hat{w}_{t-1}^{\left(i\right)}
\nu \left(\hat{\xi}_{t-1}^{\left(i\right)}, \xi_t^{\left(i\right)}\right)\right] \\
& = \Prob \left\{\max_{1\leq j\leq N_t}\frac{\hat{w}_{t-1}^{\left(j\right)}}{N_0} > \delta\right\}
+ \frac{1}{\epsilon}
\E\left[\prod_{s=0}^{t-1} G_s \left| f\right| \mathds{1}\left\{\left|f\right| \geq \frac{C}{\delta}\right\}\right]
\end{align*}
Sending $N_0$ to infinity and $\delta$ to zero verifies \eqref{weaklindeberg} at time $t$.
\end{proof}

\begin{proof}[Proof of Theorem \ref{weak}]
Lemma \ref{weaklemma} verifies the first condition of Lemma \ref{condconv}, namely,
\begin{equation*}
\sum_{t=0}^T \sum_{i=1}^{N_t} 
E\left[\left.
\frac{1}{N_0} w_t^{\left(i\right)} \left| h_t \left(\xi_t^{\left(i\right)}\right) \right|
\mathds{1}\left\{\frac{1}{N_0} w_t^{\left(i\right)} \left|h_t \left(\xi_t^{\left(i\right)}\right)\right| > C\right\}
\right|\mathcal{F}_t^{\left(i-1\right)}\right] 
\stackrel{P}{\rightarrow} 0
\end{equation*}
To verify the second condition, observe 
\begin{align*}
\label{tightness}
& \quad \Prob \left\{\sum_{t=0}^T \sum_{i=1}^{N_t} \E \left[\left.\frac{1}{N_0} w_t^{\left(j\right)} \left|h_t\left(\xi_t^{\left(j\right)}\right)\right|
\right\rvert \mathcal{F}_t^{\left(i - 1\right)}\right] > \lambda
\right\} \\
& \leq \frac{1}{\lambda} \sum_{t=0}^T \E \left[\frac{1}{N_0} \sum_{j=1}^{N_t} w_t^{\left(j\right)} 
\left|h_t\left(\xi_t^{\left(j\right)}\right)\right|\right]
\leq \frac{T+1}{\lambda} \E\left[\prod_{t=0}^{T-1} G_t \left|f\right|\right]
\end{align*}
The last quantity tends to zero as $\lambda \rightarrow \infty$.
Apply Lemma \ref{condconv} to conclude.
\end{proof}

\begin{proof}[Proof of Theorem \ref{bounded}]
The proof
uses a standard variance decomposition for martingales:
\begin{equation*}
\Var\left[M_T\right]
= \sum_{t=0}^T \sum_{i=1}^{N_t} 
\E\left[\Var\left[\left.M_t^{\left(i\right)} \right| \mathcal{F}_t^{\left(i-1\right)}\right]\right] \\
\leq \E\left[\sum_{t=0}^T \sum_{i=1}^{N_t} 
\left|\frac{1}{N_0} w_t^{\left(i\right)} h_t\left(\xi_t^{\left(i\right)}\right) \right|^2 \right]
\end{equation*}
Since functions $G_t$ are bounded, weights
$\hat{w}_t^{\left(i\right)}$ are also bounded, with $\hat{w}_t^{\left(i\right)} \leq \prod_{s=0}^t C_s \sup G_s$.
Thus, conclude
\begin{align*}
\Var\left[M_T\right] & \leq \sum_{t=0}^T \left(\prod_{s=0}^{t-1} C_s \sup G_s\right)
\E\left[ \frac{1}{N_0^2} \sum_{i=1}^{N_t} 
\hat{w}_{t-1}^{\left(i\right)}
\left|G_t \left(\hat{\xi}_{t-1}^{\left(i\right)}, \xi_t^{\left(i\right)}\right) h_t\left(\xi_t^{\left(i\right)}\right) \right|^2 \right] \\
& = \frac{1}{N_0} \sum_{t=0}^T \left(\prod_{s=0}^{t-1} C_s \sup G_s\right)
\E\left[ \prod_{s=0}^{t-1} G_s \left|G_t h_t\right|^2\right] \\
& \leq \frac{1}{N_0} \E\left[\prod_{t=0}^{T-1} G_t f^2\right] \prod_{t=0}^{T-1} \sup G_t 
\sum_{t=0}^T \prod_{s=0}^{t-1} C_s
\end{align*}
\end{proof}

The proof of Theorem \ref{multvar} requires a series of lemmas.


\begin{lem}{\label{tools}}
For each $n\geq1$, suppose $\mathcal{G}_{n0}\subseteq\mathcal{G}_{n1}\subseteq\mathcal{G}_{n2}\subseteq\cdots\subseteq\mathcal{G}_{n,k_n}$
is a filtration and $S_{n,k_n} = \sum_{j=1}^{k_n} Y_{nj}$ is the sum of martingale differences
with $\E\left[\left.Y_{nj}\right| \mathcal{G}_{n,j-1}\right] = 0$.
Define $V_{n, k_n}^2 = \sum_{j=1}^{k_n} \Var\left[\left.Y_{nj}\right| \mathcal{G}_{n,j-1}\right]$.
\begin{enumerate}[label = (\alph*)]
\item{\label{parta}}
If $V_{n, k_n}^2 \stackrel{\Prob}{\rightarrow} 1$ and if
$\sum_{j=1}^{k_n} \E\left[\left.Y_{nj}^2 \mathds{1}\left\{\left|Y_{nj}\right| > C\right\}\right| \mathcal{G}_{n,j-1}\right] \stackrel{\Prob}{\rightarrow} 0$ for each $C > 0$,
then $S_{n, k_n} \stackrel{\mathcal{D}}{\rightarrow} N\left(0, 1\right)$.
\item{\label{partb}}
If $\Prob\left\{V_{n, k_n}^2 > 1 + \epsilon\right\} \rightarrow 0$ for all $\epsilon > 0$,
then $\left|S_{n,k_n}\right| \lesssim 1$.
\end{enumerate}
\end{lem}

\begin{proof}
Part \ref{parta} is the martingale
CLT \cite[pg.58-59]{hall1980martingale}.
To prove part \ref{partb}, first set $E_{m,n} = \left\{V_{n,k_n}^2 \leq 1 + \frac{1}{m}\right\}$
for $m,n \geq 1$ and $D_{1,n} = E_{1,n}$ for $n \geq 1$.
By the assumption in part \ref{partb}, there exists a number $N\left(m\right) \geq 1$ such that $\Prob\left(E_{m,n}\right) \geq 1 - \frac{1}{m}$
for $n \geq N\left(m\right)$.
Accordingly, for $m \geq 1$ define 
\begin{equation*}
\begin{cases}
D_{m,n} = D_{m-1,n}, & n < N\left(m\right) \\
D_{m,n} = E_{m,n}, & n \geq N\left(m\right)
\end{cases}
\end{equation*}
By this construction,
for any $m \geq M$ and $n \geq N\left(M\right)$,
$\E\left[\mathds{1}_{D_{m,n}} S_{n,k_n}^2 \right] \leq 1 + \frac{1}{M}$
and $\Prob\left(D_{m,n}\right) \geq 1 - \frac{1}{M}$.
Setting $B_n = D_{n,n}$ gives $\limsup_{n \rightarrow \infty} \E\left[\mathds{1}_{B_n} S_{n,k_n}^2\right] \leq 1$
and $\Prob\left(B_n\right) \rightarrow 1$.
%
\end{proof}

\begin{lem}{\label{lemma2}}
Assume
$\E\left[\prod_{s=0}^{t} G_s \right]<\infty$ for $0 \leq t \leq T-1$
and assume $\E\left[\prod_{s=0}^t G_s \left|G_t h_t\right|^2 \right]<\infty$. 
Then for each $0 \leq t \leq T$ and $C > 0$,
\begin{equation*}
\E\left[\left. \frac{1}{N_0} \sum_{i=1}^{N_t}
\left| w_t^{\left(i\right)} h_t\left(\xi_t^{\left(i\right)}\right) \right|^2
\mathds{1}\left\{\frac{1}{\sqrt{N_0}} w_t^{\left(i\right)} 
\left| h_t\left(\xi_t^{\left(i\right)}\right) \right|\geq C\right\} 
\right\rvert\mathcal{F}_{t-1} \right]
\stackrel{\Prob}{\rightarrow} 0
\end{equation*}
\end{lem}
\begin{proof}
For the $t = 0$ case, use the Dominated Convergence Theorem.
For $1 \leq t \leq T$ and $C > 0$, define 
$\nu = \left|G_t h_t\right|^2 \mathds{1}\left\{\left|G_t h_t\right| \geq \frac{C \sqrt{N_0}}{2 \E\left[\prod_{s=0}^{t-1} G_s\right]}\right\}$.
Calculate
\begin{align*}
& \quad \Prob \left\{
\E\left[\left.\frac{1}{N_{0}} \sum_{i=1}^{N_t}
\left|w_t^{\left(i\right)} h_t\left(\xi_t^{\left(i\right)}\right) \right|^2
\mathds{1}\left\{\frac{1}{\sqrt{N_0}} 
\left|w_t^{\left(i\right)} h_t\left(\xi_t^{\left(i\right)}\right) \right| \geq C\right\} 
\right\rvert\mathcal{F}_{t-1}\right]
> \epsilon\right\} \\
& \leq 
\Prob \left\{\overline{w}_{t-1} > 2 \E\left[\prod_{s=0}^{t-1} G_s\right]\right\}
+ \frac{2}{\epsilon} \E\left[\prod_{s=0}^{t-1} G_s\right]
\E\left[\frac{\overline{w}_{t-1}}{N_0} \sum_{i=1}^{N_t} 
\nu \left(\hat{\xi}_{t-1}^{\left(i\right)}, \xi_t^{\left(i\right)}\right)\right] \\
& = \Prob \left\{\overline{w}_{t-1} > 2 \E\left[\prod_{s=0}^{t-1} G_s\right]\right\}
+ \frac{2}{\epsilon} \E\left[\prod_{s=0}^{t-1} G_s\right]
\E\left[\prod_{s=0}^{t-1} G_s \nu \right]
\end{align*}
Since $\overline{w}_{t-1} \stackrel{\Prob}{\rightarrow} \E\left[\prod_{s=0}^{t-1} G_s\right]$ by Theorem \ref{weak},
both terms vanish upon sending $N_0$ to infinity.
\end{proof}
\begin{lem}{\label{CLT}}
Assume $\E\left[\prod_{s=0}^t G_s\right] < \infty$ for $0 \leq t \leq T-1$, $\E\left|G_0 h_0\right|^2 < \infty$,
and $\E\left[\prod_{s=0}^t G_s \left|G_{t+1} h_{t+1}\right|^2\right]$ for $0 \leq t \leq T-1$.
Define
\begin{equation*}
\eta^2 = \Var\left[G_0 h_0\right] + \sum_{t=0}^{T-1} \hat{\eta}_t^2 \left[h_t\right]
+ \sum_{t=0}^{T-1} \E\left[\prod_{s=0}^t G_s\right] 
\E\left[\prod_{s=0}^t G_s \Var\left[\left.G_{t+1} h_{t+1} \right| X_t\right]\right]
\end{equation*}
where $\eta^2$ depends on numbers $\left(\hat{\eta}_t\left[h_t\right]^2\right)_{0 \leq 1 \leq T-1}$.

\begin{enumerate}[label = (\alph*)]
\item{\label{partone}}
If $N_0 \hat{V}^2_t\left[h_t\right] \stackrel{\Prob}{\rightarrow} \hat{\eta}_t^2\left[h_t\right]$ for each $0 \leq t \leq T-1$,
then 
\begin{equation*}
\sqrt{N_0} \left(\frac{\overline{w}_{T-1}}{N_0} \sum_{i=1}^{N_T} 
f\left(\hat{\xi}_{T-1}^{\left(i\right)}, \xi_T^{\left(i\right)}\right) - 
\E\left[\prod_{t=0}^{T-1} G_t f\right]\right) 
\stackrel{\mathcal{D}}{\rightarrow} N\left(0, \eta^2\right)
\end{equation*}
\item{\label{parttwo}}
If $\lim_{N_0 \rightarrow \infty} \Prob\left\{N_0 \hat{V}^2_t\left[h_t\right] \geq \hat{\eta}_t^2\left[h_t\right] + \epsilon\right\} = 0$
for each $0 \leq t \leq T-1$ and $\epsilon > 0$,
\begin{equation*}
\left|\frac{\overline{w}_{T-1}}{N_0} \sum_{i=1}^{N_T} 
f\left(\hat{\xi}_{T-1}^{\left(i\right)}, \xi_T^{\left(i\right)}\right) - 
\E\left[\prod_{t=0}^{T-1} G_t f\right]\right|
\lesssim \frac{\eta}{\sqrt{N_0}}
\end{equation*}
\end{enumerate}
\end{lem}
\begin{proof}
The proof uses Lemma \ref{tools} to analyze asymptotic behavior of the martingale $\frac{\sqrt{N_0}}{\eta} M_t^{\left(i\right)}$, where $M_t^{\left(i\right)}$ is defined in equation \eqref{martrep2}.
First compute the sum of conditional variances
\begin{align*}
& \quad \sum_{t=0}^T \sum_{i=1}^{N_t} \Var\left[\left.M_t^{\left(i\right)} \right| \mathcal{F}_{t-1}^{\left(i-1\right)}\right]
= \frac{\Var\left[G_0 h_0\right]}{N_0}
+ \sum_{t=1}^T \Var\left[\left.M_t \right| \mathcal{F}_{t-1}\right] \\
& = \frac{\Var\left[G_0 h_0\right]}{N_0}
+ \sum_{t=1}^T \left(\Var\left[\left.\E\left[\left.M_t \right| \hat{\mathcal{F}}_{t-1}\right] \right| \mathcal{F}_{t-1}\right]
+ \E\left[\left.\Var\left[\left.M_t \right| \hat{\mathcal{F}}_{t-1}\right] \right| \mathcal{F}_{t-1}\right]\right) \\
& = \frac{\Var\left[G_0 h_0\right]}{N_0}
+ \sum_{t=1}^T \Var\left[\left. \frac{\overline{w}_t}{N_0} \sum_{i=1}^{N_t} h_t\left(\xi_t^{\left(i\right)}\right) \right| \mathcal{F}_{t-1}\right] \\
& + \sum_{t=1}^T \E\left[\left.
\frac{\overline{w}_{t-1}^2}{N_0^2} \sum_{i=1}^{N_t} \Var\left[\left.G_t h_t\right| X_{t-1} = \hat{\xi}_{t-1}^{\left(i\right)}\right] \right| \mathcal{F}_{t-1}\right] \\
& = \frac{\Var\left[G_0 h_0\right]}{N_0} + \sum_{t=1}^T \hat{V}_t^2\left[h_t\right]
+ \sum_{t=1}^T \frac{\overline{w}_{t-1}}{N_0^2} \sum_{i=1}^{N_{t-1}} w_t^{\left(i\right)} \Var\left[\left.G_t h_t\right| X_{t-1} = \xi_{t-1}^{\left(i\right)}\right]
\end{align*}
Theorem \ref{weak} shows that $\overline{w}_t \stackrel{\Prob}{\rightarrow} \E\left[\prod_{s=0}^{t-1} G_s\right]$
and 
\begin{equation*}
\frac{1}{N_0} \sum_{i=1}^{N_{t-1}} w_t^{\left(i\right)} \Var\left[\left.G_t h_t\right| X_{t-1} = \xi_{t-1}^{\left(i\right)}\right]
\stackrel{\Prob}{\rightarrow}
\E\left[\prod_{s=0}^{t-1} G_s \Var\left[\left.G_t h_t \right| X_{t-1}\right]\right]
\end{equation*}
Next, for $C > 0$ and $0 \leq t \leq T$, a useful inequality of \citet{dvoretzky1972asymptotic} gives 
\begin{align*}
& \quad N_0 \sum_{i=1}^{N_t} \E\left[\left.
\left|M_t^{\left(i\right)} - M_t^{\left(i-1\right)}\right|^2
\mathds{1}\left\{\sqrt{N_0}\left|M_t^{\left(i\right)} - M_t^{\left(i-1\right)}\right| \geq C\right\} 
\right| \mathcal{F}_t^{\left(i-1\right)}\right] \\
& \leq 4
\E\left[\left. \frac{1}{N_0} \sum_{i=1}^{N_t}
\left| w_t^{\left(i\right)} h_t\left(\xi_t^{\left(i\right)}\right) \right|^2
\mathds{1}\left\{\frac{1}{\sqrt{N_0}} w_t^{\left(i\right)} 
\left| h_t\left(\xi_t^{\left(i\right)}\right) \right|\geq \frac{C}{2}\right\} 
\right\rvert\mathcal{F}_{t-1} \right]
\end{align*}
This last term vanishes upon sending $N_0$ to infinity by Lemma \ref{lemma2}. 
Thus, the conditions of Lemma \ref{tools} are satisfied, and parts \ref{partone} and \ref{parttwo} follow.
\end{proof}

\begin{lem}{\label{technical1}}
Assume $\E\left[\prod_{s=0}^t G_s\right] < \infty$ for $0 \leq t \leq T-1$ and assume
$\E\left|\prod_{t=0}^T G_t f\right| < \infty$.
Set
$\tilde{G}_T = \E\left[\prod_{t=0}^{T-1} G_t \right] G_T \slash \E\left[\prod_{t=0}^T G_t\right]$.
Then
\begin{multline}{\label{awful1}}
\frac{\overline{w}_{T-1}}{N_0} \sum_{i=1}^{N_T} 
\left\{\frac{w_T^{\left(i\right)}}{\overline{w}_T}\right\} 
\left(1 - \left\{\frac{w_T^{\left(i\right)}}{\overline{w}_T}\right\}\right)
f\left(\xi_t^{\left(i\right)}\right) \\
\stackrel{\Prob}{\rightarrow} 
\E\left[\prod_{t=0}^{T-1} G_t \left\{\tilde{G}_T\right\}\left(1 - \left\{\tilde{G}_T\right\}\right) f \right]
\end{multline}
If additionally $\E\left[\prod_{t=0}^{T-1} G_t \mathds{1}\left\{\tilde{G}_T \in \left\{1, 2, \ldots\right\}\right\}\right] = 0$,
then
\begin{equation}{\label{awful2}}
\frac{\overline{w}_{T-1}}{N_0} \sum_{i=1}^{N_T} 
\left\{\frac{w_T^{\left(i\right)}}{\overline{w}_T}\right\}
f\left(\xi_T^{\left(i\right)}\right) 
\stackrel{\Prob}{\rightarrow} 
\E\left[\prod_{t=0}^{T-1} G_t \left\{\tilde{G}_T\right\} \nu\right]
\end{equation}
\end{lem}
\begin{proof}
For a proof of equation \eqref{awful2} and a special case of equation \eqref{awful1}, see \citet{douc2008limit}. 
To prove the more general case of \eqref{awful1},
first define
$L\left(x\right) = \left\{x\right\} - \left\{x\right\}^2$.
By Theorem \ref{weak},
\begin{equation*}
\frac{\overline{w}_{T-1}}{N_0} \sum_{i=1}^{N_T}
L\left(\tilde{G}_T\left(\hat{\xi}_{T-1}^{\left(i\right)}, \xi_T^{\left(i\right)}\right)\right)
f\left(\xi_T^{\left(i\right)}\right)
\stackrel{\Prob}{\rightarrow}
\E\left[\prod_{t=0}^{T-1} G_t \left\{\tilde{G}_T\right\}\left(1 - \left\{\tilde{G}_T\right\}\right) f\right]
\end{equation*}
Thus, it suffices to show
\begin{equation*}
\frac{\overline{w}_{T-1}}{N_0} \sum_{i=1}^{N_T}
\left|L\left(\frac{w_T^{\left(i\right)}}{\overline{w}_T}\right) 
- L\left(\tilde{G}_T\left(\hat{\xi}_{T-1}^{\left(i\right)}, \xi_T^{\left(i\right)}\right)\right)\right|
\left|f\left(\xi_T^{\left(i\right)}\right)\right|
\stackrel{\Prob}{\rightarrow} 0
\end{equation*}
Since $x \mapsto L\left(x\right)$ has Lipschitz constant $1$,
for $\epsilon > 0$ and $\delta > 0$ it follows that
\begin{align*}
& \quad \Prob \left\{
\frac{\hat{w}_{t-1}}{N_0} \sum_{i=1}^{N_t}
\left|L\left(\frac{w_t^{\left(j\right)}}{\hat{w}_t}\right) - 
L\left(\tilde{G}_t\left(\hat{\xi}_{t-1}^{\left(i\right)}, \xi_t^{\left(i\right)}\right)\right)\right|
\left|f\left(\xi_t^{\left(i\right)}\right)\right|
 > \epsilon \right\} \\
& \leq \Prob \left\{
\left|\frac{\overline{w}_{T-1}}{\overline{w}_T}
\frac{\E\left[\prod_{t=0}^T G_t\right]}{\E\left[\prod_{t=0}^{T-1} G_t\right]} - 1\right| > \delta \right\} \\
& + \frac{\delta}{\epsilon} 
\E \left[\frac{\hat{w}_{T-1}}{N_0} \sum_{i=1}^{N_T}
\tilde{G}_T\left(\hat{\xi}_{T-1}^{\left(i\right)}, \xi_T^{\left(i\right)}\right)
\left|f\left(\xi_T^{\left(i\right)}\right)\right|\right] \\
& = \Prob \left\{
\left|\frac{\overline{w}_{T-1}}{\overline{w}_T}
\frac{\E\left[\prod_{t=0}^T G_t\right]}{\E\left[\prod_{t=0}^{T-1} G_t\right]} - 1\right| > \delta \right\}
+ \frac{\delta}{\epsilon} 
\E \left[\prod_{s=0}^{t-1} G_s \tilde{G}_t f \right]
\end{align*}
Both terms vanish upon sending $N_0$ to infinity and then $\delta$ to $0$.
\end{proof}

\begin{lem}{\label{technical2}}
Assume $\E\left[\prod_{s=0}^t G_s\right] < \infty$ for $0 \leq t \leq T-1$
and assume $\E\left[\prod_{t=0}^T G_t f^2\right] < \infty$.
At resampling step $T$, assume particles are sorted by a coordinate $\theta_T$
and then stratified or stratified residual resampling is used.
Then for any $p \colon \mathbb{R} \rightarrow \mathbb{R}$
with $\E\left[\prod_{t=0}^T G_t \left|p\left(\theta_T\right)\right|^2\right] < \infty$,
\begin{align*}
& \limsup_{N_0 \rightarrow \infty} \Prob\left\{N_0 \hat{V}_T^2\left[f\right]
> \left(1 + \epsilon\right) N_0 \hat{V}_T^2\left[f - p\left(\theta_T\right)\right] + \epsilon \right\} < \epsilon,
& \epsilon > 0
\end{align*}
\end{lem}
\begin{proof}
Fix $\delta > 0$ and select $\eta \in C_c\left(\mathbb{R}\right)$, which approximates $p$ so that
$\E \left[\prod_{t=0}^T G_t \left|\eta \left(\theta_T\right) - p\left(\theta_T\right)\right|^2\right] < \delta$.
Applying Cauchy's inequality with $\epsilon$,
\begin{equation*}
\hat{V}_T^2\left[f\right] \leq
\left(2 + \frac{2}{\epsilon}\right) 
\left(\hat{V}_T^2\left[\eta\left(\theta_T\right)\right] 
+ \hat{V}_T^2\left[\eta\left(\theta_T\right) - p\left(\theta_T\right)\right]\right)
+ \left(1 + \epsilon\right) \hat{V}_T^2\left[f - p\left(\theta_T\right)\right]
\end{equation*}
To prove the result it suffices to bound
$\hat{V}_T^2\left[\eta\left(\theta_T\right)\right]$
and $\hat{V}_T^2\left[\eta\left(\theta_T\right) - p\left(\theta_T\right)\right]$.
First bound $\hat{V}_T^2\left[\eta\left(\theta_T\right)\right]$.
On the event that $\hat{w}_T \leq 2\E\left[\prod_{t=0}^T G_t\right]$, it follows
\begin{align*}
& \quad \hat{V}_T^2\left[\eta\left(\theta_T\right)\right] 
\leq \left(2\E\left[\prod_{t=0}^T G_t\right]\right)^2
\Var\left[\left.\frac{1}{N_0} \sum_{j=1}^{N_{T+1}} \eta\left(\hat{\theta}_T^{\left(j\right)}\right)
\right| \mathcal{F}_T\right] \\
& = \left(2\E\left[\prod_{t=0}^T G_t\right]\right)^2
\frac{1}{N_0^2} \sum_{j=1}^{N_{T+1}} \Var\left[\left. \eta\left(\hat{\theta}_T^{\left(j\right)}\right)
\right| \mathcal{F}_T\right]
\end{align*}
where $\hat{\theta}_T^{\left(j\right)}$ denotes $\theta_T\left(\hat{\xi}_T^{\left(j\right)}\right)$.
In the resampling step,
a series of particles $\left(\hat{\xi}_T^{\left(j\right)}\right)_{1 \leq j \leq J}$ is randomly selected
(other particles may be deterministically selected)
with $L^{\left(0\right)} \geq \hat{\theta}_T^{\left(1\right)} \geq L^{\left(1\right)} \geq \hat{\theta}_T^{\left(2\right)} \geq \cdots \geq \hat{\theta}_T^{\left(J\right)} \geq L^{\left(J\right)}$ for some
$\mathcal{F}_T$-measurable random variables $L^{\left(j\right)}$. Therefore,
\begin{align*}
& \quad \sum_{j=1}^J \Var\left[\left.\eta\left(\hat{\theta}_T^{\left(j\right)}\right)\right| \mathcal{F}_T\right]
\leq \frac{1}{4} \sum_{j=1}^J \left|\sup_{x \in \left[L^{\left(j-1\right)}, L^{\left(j\right)}\right]} \eta\left(x\right)
- \inf_{x \in \left[L^{\left(j-1\right)}, L^{\left(j\right)}\right]} \eta\left(x\right)\right|^2 \\
& \leq \sup_{x \in \mathbb{R}}\left|\eta\left(x\right)\right| \sum_{j=1}^J 
\left|\sup_{x \in \left[L^{\left(j-1\right)}, L^{\left(j\right)}\right]} \eta\left(x\right)
- \inf_{x \in \left[L^{\left(j-1\right)}, L^{\left(j\right)}\right]} \eta\left(x\right)\right|
\leq \sup_{x \in \mathbb{R}}\left|\eta\left(x\right)\right| V\left(\eta\right)
\end{align*}
where $V\left(\eta\right)$ is the total variation of $\eta$.

It remains to bound $\hat{V}_T^2\left[\eta\left(\theta_T\right) - p\left(\theta_T\right)\right]$.
On the event  $\hat{w}_T \leq 2\E\left[\prod_{t=0}^T G_t\right]$,
\begin{align*}
& \quad \hat{V}_T^2\left[\eta\left(\theta_T\right) - p\left(\theta_T\right)\right]
= \sum_{j=1}^{N_{T+1}} \frac{\hat{w}_T^2}{N_0^2}
\Var\left[\left. 
\left(\eta\left(\hat{\theta}_T^{\left(j\right)}\right) - p\left(\hat{\theta}_T^{\left(j\right)}\right)\right)\right|\mathcal{F}_T\right] \\
& \leq 2\E\left[\prod_{t=0}^T G_t\right]
\E\left[\left. \frac{\hat{w}_T}{N_0} \sum_{j=1}^{N_{T+1}} \left|\eta\left(\hat{\theta}_T^{\left(j\right)}\right) - p\left(\hat{\theta}_T^{\left(j\right)}\right)\right|^2\right|\mathcal{F}_T\right]
\end{align*}
This last term has expectation
\begin{equation*}
2\E\left[\prod_{t=0}^T G_t\right] \E \left[\prod_{t=0}^T G_t \left|\eta \left(\theta_T\right) - p\left(\theta_T\right)\right|^2\right] 
< 2 \delta \E\left[\prod_{t=0}^T G_t\right]
\end{equation*}
Conclude
\begin{align*}
& \quad \limsup_{N_0 \rightarrow \infty} \Prob\left\{N_0 \hat{V}_t^2\left[h_t\right]
> \left(1 + \epsilon\right) N_0 \hat{V}_t^2\left[h_t - \nu\right] + \epsilon \right\} \\
& = \limsup_{N_0 \rightarrow \infty} 
\Prob\left\{\hat{w}_t \leq 2\E\left[\prod_{s=0}^t G_s\right],
\, \left(2 + \frac{2}{\epsilon}\right)N_0 \hat{V}_T^2\left[\eta\left(\theta_T\right)\right] > \frac{\epsilon}{2}\right\} \\
& + \limsup_{N_0 \rightarrow \infty}
\Prob\left\{\hat{w}_t \leq 2\E\left[\prod_{s=0}^t G_s\right],
\, \left(2 + \frac{2}{\epsilon}\right) N_0 \hat{V}_T^2\left[\eta\left(\theta_T\right) - p\left(\theta_T\right)\right] > \frac{\epsilon}{2}\right\} \\
& \leq \left(\frac{2}{\epsilon}\right)\left(2 + \frac{2}{\epsilon}\right)
\left(2 \delta \E\left[\prod_{t=0}^T G_t\right]\right)
\end{align*}
For small enough $\delta$, this last term is less than $\epsilon$, proving the result.
\end{proof}

\begin{proof}[Proof of Theorem \ref{multvar}]
The proof combines Lemma \ref{CLT} with explicit computations of resampling variances $\hat{V}_t^2$.
For multinomial resampling,
%
\begin{equation*}
N_0 \hat{V}_t^2\left[h_t\right] = \frac{\overline{w}_t}{N_0} \sum_{i=1}^{N_t} w_t^{\left(i\right)} 
\left|h_t\left(\xi_t^{\left(i\right)}\right)\right|^2
- \left|\frac{1}{N_0} \sum_{i=1}^{N_t} w_t^{\left(i\right)} 
h_t\left(\xi_t^{\left(i\right)}\right)\right|^2
\end{equation*}
By Theorem \ref{weak}, therefore,
\begin{multline*}
N_0 \hat{V}_t^2\left[h_t\right] \stackrel{\Prob}{\rightarrow} \E\left[\prod_{s=0}^t G_s\right] \E\left[\prod_{s=0}^t G_s h_t^2\right]
- \left(\E\left[\prod_{s=0}^t G_s h_t \right]\right)^2 \\
= \E\left[\prod_{s=0}^t G_s\right] \E\left[\prod_{s=0}^t G_s 
\left|h_t - \frac{\E\left[\prod_{s=0}^t G_s h_t\right]}{\E\left[\prod_{s=0}^t G_s\right]}\right|^2\right]
\end{multline*}
For multinomial residual resampling, $N_0 \hat{V}_t^2\left[h_t\right]$ takes the form
\begin{equation*}
\frac{\overline{w}_t^2}{\overline{w}_{t-1}} \left(
\frac{\overline{w}_{t-1}}{N_0} \sum_{i=1}^{N_t} \left\{\frac{w_t^{\left(i\right)}}{\overline{w}_t}\right\}
\left|h_t\left(\xi_t^{\left(i\right)}\right)\right|^2
- \frac{
\left|\frac{\overline{w}_{t-1}}{N_0} \sum_{i=1}^{N_t} \left\{\frac{w_t^{\left(i\right)}}{\overline{w}_t}\right\}
h_t\left(\xi_t^{\left(i\right)}\right)\right|^2}
{\frac{\overline{w}_{t-1}}{N_0}\sum_{i=1}^{N_t} \left\{\frac{w_t^{\left(i\right)}}{\overline{w}_t}\right\}}\right)
\end{equation*}
By Theorem \ref{weak} and Lemma \ref{technical1}, $N_0 \hat{V}_t^2\left[h_t\right]$ converges in probability to
\begin{equation*}
\frac{\left(\E\left[\prod_{s=0}^t G_s\right]\right)^2}{\E\left[\prod_{s=0}^{t-1} G_s\right]}
\E\left[\prod_{s=0}^{t-1} G_s \left\{\tilde{G}_t\right\} \left|h_t
- \frac{\E\left[\prod_{s=0}^{t-1} G_s \left\{\tilde{G}_t\right\} h_t\right]}
{\E\left[\prod_{s=0}^{t-1} G_s \left\{\tilde{G}_t\right\}\right]}\right|^2\right]
\end{equation*}
For Bernoulli resampling, Theorem \ref{weak} and Lemma \ref{technical1} give
\begin{align*}
N_0 \hat{V}_t^2\left[h_t\right] & = \frac{\overline{w}_t^2}{N_0} \sum_{i=1}^{N_t} 
\left\{\frac{w_t^{\left(i\right)}}{\overline{w}_t}\right\}
\left(1 - \left\{\frac{w_t^{\left(i\right)}}{\overline{w}_t}\right\}\right)
\left|h_t\left(\xi_t^{\left(i\right)}\right)\right|^2 \\
& \stackrel{\Prob}{\rightarrow}
\frac{\left(\E\left[\prod_{s=0}^t G_s\right]\right)^2}{\E\left[\prod_{s=0}^{t-1} G_s\right]}
\E\left[\prod_{s=0}^{t-1} G_s \left\{\tilde{G}_t\right\}\left(1 - \left\{\tilde{G}_t\right\}\right) h_t^2\right]
\end{align*}
To compute the resampling variance for stratified resampling, consider the function
$p$ that minimizes
$\E\left[\prod_{s=0}^t \tilde{G}_s \left|h_t - p\left(\theta_t\right)\right|^2\right]$.
Since this function can be written as an $L^2$ projection, it is well-defined.
Moreover, by Lemma \ref{simplelemma}, 
the resampling variance $N_0\hat{V}_t^2\left[h_t - p\left(\theta_t\right)\right]$
is bounded by the multinomial resampling variance,
which converges in probability to 
\begin{equation*}
\hat{\eta}_t^2\left[h_t\right] = \left(E\left[\prod_{s=0}^t G_s\right]\right)^2
\E\left[\prod_{s=0}^t \tilde{G}_s \left|h_t - p\left(\theta_t\right)\right|^2\right]
\end{equation*}
Thus, $\Prob\left\{N_0\hat{V}_t^2\left[h_t - p\left(\theta_t\right)\right] > \hat{\eta}_t^2\left[h_t\right] + \epsilon\right\} \rightarrow 0$ for all $\epsilon > 0$.
By Lemma \ref{technical2}, this is enough to guarantee
$\Prob\left\{N_0 \hat{V}_t^2\left[h_t\right] > \hat{\eta}_t^2\left[h_t\right] + \epsilon\right\} \rightarrow 0$ for all $\epsilon > 0$.
The asymptotic variance upper bound for sorted stratified residual resampling is proved similarly.
\end{proof}


\section*{Acknowledgements}
The author would like to thank Jonathan Weare and Omiros Papaspiliopoulos
for conversations that helped shape the presentation of results 
and Alicia Zhao for gracious and patient
editorial assistance.



\end{document}